\documentclass[reqno, 11pt, letterpaper]{amsart}

\usepackage{amsmath}
\usepackage{amsfonts}
\usepackage{amssymb}
\usepackage{graphicx}
\usepackage{amsthm,graphicx,color,yfonts}
\usepackage{array} 
\usepackage{hyperref}
\usepackage{multirow}

\newtheorem{theorem}{Theorem}

\newtheorem{proposition}[theorem]{Proposition}
\newtheorem{lemma}[theorem]{Lemma}

\theoremstyle{remark}
\newtheorem{remark}[theorem]{Remark}

\newcommand{\ls}{\lesssim}

\newcommand{\R}{\mathbb{R}}
\newcommand{\T}{\mathbb{T}}
\newcommand{\C}{\mathbb{C}}
\newcommand{\Z}{\mathbb{Z}}
\newcommand{\N}{\mathbb{N}}

\def\norm#1{\left\|#1\right\|}

\usepackage{wrapfig}
\usepackage{tikz}
\usetikzlibrary{arrows,calc,decorations.pathreplacing}
\definecolor{light-gray1}{gray}{0.90}
\definecolor{light-gray2}{gray}{0.80}
\definecolor{light-gray3}{gray}{0.60}

\numberwithin{equation}{section}

\numberwithin{theorem}{section}

\numberwithin{table}{section}

\numberwithin{figure}{section}

\ifx\pdfoutput\undefined
  \DeclareGraphicsExtensions{.pstex, .eps}
\else
  \ifx\pdfoutput\relax
    \DeclareGraphicsExtensions{.pstex, .eps}
  \else
    \ifnum\pdfoutput>0
      \DeclareGraphicsExtensions{.pdf}
    \else
      \DeclareGraphicsExtensions{.pstex, .eps}
    \fi
  \fi
\fi

\title[Continuum limit for DNLS on a large finite lattice]{On the continuum limit for the discrete Nonlinear Schr\"odinger equation on a large finite cubic lattice}


\author[Y. Hong]{Younghun Hong}
\address{Department of Mathematics, Chung-Ang University, Seoul 06974, Korea}
\email{yhhong@cau.ac.kr}

\author[C. Kwak]{Chulkwang Kwak}
\address{Department of Mathematics, Ewha Womans University, Seoul 03760, Korea}
\email{ckkwak@ewha.ac.kr}

\author[C. Yang]{Changhun Yang}
\address{Department of Mathematics, Chungbuk National University, Cheongju-si, Chungcheongbuk-do 28644, Korea}
\email{chyang@chungbuk.ac.kr}

\begin{document}	

\maketitle

\begin{abstract}
In this study, we consider the nonlinear Sch\"odinger equation (NLS) with the zero-boundary condition on a two- or three-dimensional large finite cubic lattice. We prove that its solution converges to that of the NLS on the entire Euclidean space with simultaneous reduction in the lattice distance and expansion of the domain. Moreover, we obtain a precise global-in-time bound for the rate of convergence. Our proof heavily relies on Strichartz estimates on a finite lattice. A key observation is that, compared to the case of a lattice with a fixed size \cite{Hong2021}, the loss of regularity in Strichartz estimates can be reduced as the domain expands, depending on the speed of expansion. This allows us to address the physically important three-dimensional case.
\end{abstract}
\section{Introduction}\label{sec:1}

\subsection{Setup of the problem}\label{sec:1.1}

Let $d=2,3$. For large integers $K, R\in\mathbb{N}$, we let $h=\frac{\pi}{K}$ and consider a finite cubic lattice domain of the form
%\begin{equation}\label{eq:Omega}
\[\begin{aligned}
\Omega=\Omega_{h,\pi R}:&=\Big\{x=(x_1,\cdots,x_d)\in h\mathbb{Z}^d : -\pi R\leq x_j\leq\pi R\Big\}\\
&=\Big\{x=hm=(hm_1,\cdots, hm_d): m_j=-KR, ..., -2, -1,0,1, 2, ...,KR\Big\}
\end{aligned}\]
%\end{equation}
whose boundary is denoted by
\begin{equation}\label{eq:boundary}
\partial\Omega:=\Big\{x\in\Omega: x_j=\pm\pi R\textup{ for some }j\Big\}.
\end{equation}
Here, we assume that 
\begin{equation}\label{eq:DomainR}
R\sim h^{-\alpha}\quad\textup{for some }\alpha>0.
\end{equation}
In other words, our domain is a collection of order $\frac{1}{h^{(1+\alpha)d}}$ many points with grid size $h$, contained in a $d$-dimensional large cube with side length comparable to $h^{-\alpha}$.

In this domain, the Laplacian operator $\Delta_\Omega$ is defined as a discrete Laplacian acting on functions with zero boundary values. Precisely, for a function $f:\Omega\to\mathbb{C}$ such that $f|_\Omega\equiv 0$, we define 
\begin{equation}\label{eq:LaplaceOmega}
\Delta_\Omega f(x) = \left\{\begin{aligned}&\frac{1}{h^2}\sum_{j=1}^df(x+h\mathbf{e}_j)+f(x-h\mathbf{e}_j)-2f(x) &&\textup{for } x \in \Omega\setminus\partial\Omega, \\ &0 &&\textup{for }x \in\partial\Omega,\end{aligned}\right.
\end{equation}
where $\{\mathbf{e}_1,\cdots, \mathbf{e}_d\}$ is the standard basis for $\mathbb{R}^d$. Note that at least formally, as $h\to 0$, the finite lattice $\Omega$ shall cover the entire Euclidean space $\mathbb{R}^d$, while the discrete Laplacian $\Delta_\Omega$ becomes the Laplacian $\Delta$ on $\mathbb{R}^d$ acting on functions decreasing to zero at infinity.

In this study, we focus on the nonlinear Schr\"odinger equation (NLS) on a finite lattice,
\begin{equation}\label{eq:DRNLS}
\left\{\begin{aligned}
i\partial_t w +\Delta_{\Omega}w - |w|^2w&=0,\\
w(0) &= w_0,
\end{aligned}\right.
\end{equation}
where
$$w=w(t,x):\mathbb{R}\times\Omega\to\mathbb{C},$$
and we are concerned with a rigorous justification of its continuum limit (as $h\to 0$) towards the NLS on the entire Euclidean space, 
\begin{equation}\label{eq:NLS}
\left\{\begin{aligned}
i\partial_t u +\Delta u - |u|^2u&=0,\\
u(0) &= u_0,
\end{aligned}\right.
\end{equation}
where
$$u=u(t,x):\mathbb{R}\times\mathbb{R}^d\to\mathbb{C}.$$
Note that in this article, we will restrict ourselves to solutions to the NLS \eqref{eq:NLS} in a weighted $L^2$-space (see Theorem \ref{thm:main}). Therefore, it is natural to compare them with solutions to the truncated equation \eqref{eq:DRNLS} with the zero boundary condition.

The rationale behind conducting this study is as follows. First, for numerical simulation, the finite system of nonlinear ODEs \eqref{eq:DRNLS} is generated from the PDE \eqref{eq:NLS} by the semi-discrete approach. Proving the continuum limit validates the numerical scheme. In contrast, in some physical context such as solid state physics and condensed matter physics, a lattice model may arise as a tight-binding approximation for the continuous NLS with a periodic potential. Then, the effective dynamics in the small amplitude regime is described by the continuous model (see \cite{Kev}). Indeed, with simple scaling, the continuum limit can be translated into a small-amplitude limit on a lattice with a fixed grid size, and vice versa (see Appendix \ref{sec:small amplitude limit} for details). Since these two limits are mathematically equivalent, only the former numerical aspect shall be discussed hereafter. 

The NLS is a canonical equation for nonlinear wave propagation, and it appears in various areas of physics. For instance, it describes laser-beam propagation in nonlinear optical media and the mean-field dynamics of Bose--Einstein condensate (see \cite{SulemSulem}). From an analysis perspective, the NLS is one of the simplest prototypical equations in the class of nonlinear dispersive equations. Thus, beside its own physical importance (see \cite{Kev, KEVREKIDIS2001}), the discrete NLS would be a natural first choice to explore general nonlinear dispersive phenomena in a discrete setting. \\

For a mathematical formulation of the continuum limit, we need a few more definitions. For a function $f:\mathbb{R}^d\to\mathbb{C}$, its discretization is defined by
\begin{equation}\label{discretization}
(d_hf)(x):=\frac{1}{h^d}\int_{[0,h)^d}f(x+y)dy\quad\textup{for all }x\in h\mathbb{Z}^d,
\end{equation}
where $h\mathbb{Z}^d=\{hm: m=(m_1,\cdots, m_d)\in\mathbb{Z}^d\}$ is a finite cubic lattice. For localization, we fix a smooth bump function $\eta:\mathbb{R}^d\to[0,1]$ such that $\eta\equiv1$ if $|x|\leq 1$ and $\eta\equiv0$ if $|x|\geq 2$, and let
$$\eta_R(x)=\eta(\tfrac{x}{R}).$$
Then, given the initial data $u_0\in H^1(\mathbb{R}^d)$, a localized discrete function 
$$w_0(x)=\eta_R(x)(d_hu_0)(x):\Omega\to\mathbb{C}$$
is obtained as initial data for the discrete model. Let $u(t)$ be the solution to the NLS \eqref{eq:NLS} with initial data $u_0$, and let $w(t)$ be the solution to the discrete NLS \eqref{eq:DRNLS} with initial data $\eta_R(d_hu_0)$. By the conservation laws, the NLS is globally well-posed in $H^1(\mathbb{R}^d)$ (see \cite{Cazenave}), and so is the discrete NLS (see Section \ref{sec:6}). Thus, both $u(t)$ and $w(t)$ exist globally in time.

To compare these two global solutions given on different domains, we define the trivial extension operator $\mathcal{E}$ as
\begin{equation}\label{eq: extension}
\mathcal{E}f(x)=\left\{\begin{aligned}
&f(x) &&\textup{if }x\in\Omega,\\
&0 &&\textup{if }x\in h\mathbb{Z}^d\setminus\Omega
\end{aligned}\right.
\end{equation}
for $f:\Omega\to \mathbb{C}$, and the linear interpolation as
\begin{equation}\label{linear interpolation}
(\ell_h g)(x):=g(y)+\sum_{j=1}^d\frac{f(y+h\mathbf{e}_j)-f(y)}{h}(x_j-y_j)\quad\textup{for all }x\in y+[0,h)^d\textup{ with }y\in h\mathbb{Z}^d
\end{equation}
for $g:h\mathbb{Z}\to\mathbb{C}$. Then, we aim to show that
$$\lim_{h\to0}\ell_h\mathcal{E}w(t)=u(t),$$
which can be visualized as the diagram presented below.

\begin{figure}[h!]
\begin{center}
\begin{tikzpicture}[scale=0.5]
\node at (-7.25,-4){$w_0$};
\draw[thick,->] (-4.25,-4) -- (4.25,-4);
\node at (-7.25,3.75){\boxed{\text{initial data}}};
\node at (-7.25,2){$u_0$};
\draw[thick,->] (-4.25,2) -- (4.25,2);
\node at (0,3.75){\text{nonlinear evolution}};
\draw[thick,->] (-7.25,0.65) -- (-7.25,-2.65);
\node at (-0.5,-1){\Huge$\circlearrowleft$};
\node at (-9.5,0){{\small Discretization \ \ }};
\node at (-9.5,-1){{\small and  \ \ }};
\node at (-9.5,-2){{\small Localization  \ \ }};
\node at (7.5,-4){$w(t)$};
\node at (13.5,-4){\boxed{\text{NLS on } \Omega}};
\node at (7.5,3.75){\boxed{\text{solution}}};
\node at (7.5,2){$u(t)$};
\node at (13.5,2){\boxed{\text{NLS on } \R^d}};
\node at (6.8,-1.15){?};
\draw[thick,dashed, <-] (7.5,0.65) -- (7.5,-2.65);
\node at (9,-0.5){$ h \to 0$};
\node at (10.4,-1.5){$R = h^{-\alpha} \to \infty$};
\end{tikzpicture}
\end{center}
\caption{Schematic description of the continuum limit on a large finite lattice}\label{Fig:Question}
\end{figure}
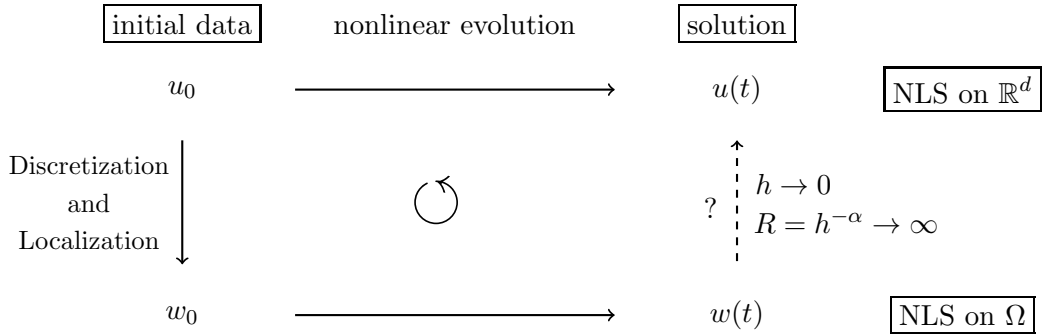

The continuum limit of a discrete NLS has been studied by many authors in various settings. First of all, on an infinite cubic lattice $h\Z^d$, Ignat and Zuazua \cite{IZ-CR05-01,IZ-SIAM09,IZ-JMPA12} and the first and third authors \cite{Hong2019a} showed  that general solutions to the discrete NLS converge to those to the continuous NLS on $\mathbb{R}^d$. On a finite periodic lattice $h\Z^2\setminus2\pi\Z^2$, the continuum limit to the 2d periodic NLS is established in \cite{Hong2021}. Considering long range interactions, a fractional NLS is obtained from the continuum limit of a discrete model \cite{V-2006,KLS, Hong2019a, Grande2019}. A relativistic Dirac equation is derived from a Schr\"odinger equation on a hexagonal honeycomb lattice, describing a quantum particles on a graphene \cite{ANZ-2009, AZ-2012, Arbunich2018}.  We also note that continuum limits  have been studied focusing on ground states and their stability \cite{Bambusi2010,JW1,JW2,Bernier2019,Bernier2019a}.

Beside the NLS-type equation, continuum limit problems can be formulated for different types of dispersive equations. For instance, it is a historically important discovery that the Korteweg-de Vries equation emerges from the Fermi-Pasta-Ulam system \cite{FPUT-1955, ZK-1965}. This can be rigorously justified either as a small amplitude or as a continuum limit (see \cite{Schneider2000, BP-2006, Hong2021a} for general solutions, and \cite{FP,M-2013} for solitary waves).

We note that for continuum limits of dispersive equations, a crucial task is to obtain \textit{uniform} (in mesh size) Strichartz estimates. Indeed, a remarkable feature of discrete models is that the dispersion properties differ from those for the continuous ones. For the Schr\"odinger-type models, the dispersion effect is weaker in a discrete setting \cite{SK, Hong2019}, while discrete wave-type equations may enjoy stronger dispersion \cite{Schultz1998, BG}.

%{\color{red}From a numerical perspective, the \textit{fully discretized} model (spatial as well as time discretization) has been broadly considered.  It is a classical problem in numerical analysis and has been extensively researched (e.g., Serna-Verwer \cite{SANZSERNA1986}, Akrivis-Dougailis-Karakashian \cite{Akrivis1991}), where typical time-discretization schemes were introduced. We also refer to a report by Ignat \cite{Ignat07}, where the time-splitting method for general solutions on $h\Z$ was treated, and a report by Bambusi-Faou-Grébert \cite{BFG} on the stability of ground states of the fully discrete NLS with Dirichlet boundary condition. Finally, we refer to a survey article by Bao-Cai \cite{Bao2013}. In the future, we shall aim to find an approximate rate of general solutions to the fully discrete model with Dirichlet boundary condition. }

%{\color{purple}In the area of numerical analysis, various numerical methods for computing the dynamics and solitons to NLS via \textit{full discretization}(spatial as well as time) have been developed over several decades. The continuum limit of such fully discretized NLS has been intensively studied \cite{SANZSERNA1986,Weideman-1986,Akrivis1991,Glassey-1992,Ignat07,Bao2013,BFG} to provide rigorous error estimates for approximation, which plays a role as an benchmark to test the accuracy, efficiency and stability of numerical schemes. }

Finally, we mention that from a numerical analysis perspective, it would be desirable to prove convergence of fully discretized model (including the time discretization), because solving a finite nonlinear system such as  \eqref{eq:DRNLS} may be computationally very expensive. In this context, we refer to \cite{SANZSERNA1986, Weideman-1986, Akrivis1991, Glassey-1992, Ignat07, Bao2013, BFG} and to the survey article of Bao and Cai \cite{Bao2013}.

% For stability and convergence of soliton solutions to the KdV equation, we refer to \cite{FP,M-2013}. 
%Furthermore, with a different scaling, the NLS is derived from the FPU system \cite{S-2010}.

\subsection{Main result}

Now, we state the main theorem of the paper which establishes the continuum limit in three dimensions, with a precise bound on the rate of convergence for an arbitrarily long period of time. 

\begin{theorem}[Continuum limit]\label{thm:main}
Let $d=2,3$. For $h\in(0,1]$ with $K = \frac{\pi}{h}\in\mathbb{N}$, we take $R \sim h^{-\alpha}$ for some $\alpha > 0$. Given $u_0\in H^{1,1}(\mathbb{R}^d)$\footnote{The weighted Sobolev space $H^{1,1}$ will be precisely defined as in \eqref{eq:weighted Sobolev}.}, let $u(t)$ be the global solution to NLS \eqref{eq:NLS} with initial data $u_0$, and let $w(t)$ be the global solution to discrete NLS \eqref{eq:DRNLS} with initial data $\eta_R(d_hu_0)$. Then, there exists $C, c>0$, depending only on $\|u_0\|_{H^{1,1}(\mathbb{R}^d)}$ and not on $h$ such that
\begin{equation}\label{eq:main result}
\|u(t)-\ell_h\mathcal{E}w(t)\|_{L^2(\mathbb{R}^d)}\leq Ch^{\min\{\alpha,\frac{1}{2}\}} e^{c t}\quad\textup{for all }t\in\mathbb{R}.
\end{equation}
\end{theorem}

To best of our knowledge, it is the first attempt to study the continuum limit from a finite system to the NLS on an unbounded domain. Our approach is robust. As mentioned earlier, we choose the NLS \eqref{eq:DRNLS} for concreteness and simplicity of discussion. We expect that it can be applied to finite systems of other dispersive equations. Furthermore, some of the ideas and analysis tools in this article could be employed prove the continuum limit of more practical fully discretized models (see \cite{Bao2013}). It is important from a numerical analysis perspective, however it  will be left to future work.

Note also that for global-in-time convergence in multi-dimensional cases, it is required to exploit additional regularization properties arising from the dispersive effect\footnote{In the one-dimensional case, or if initial data are sufficiently regular and the interval of convergence is sufficiently short, such a convergence can be proved  simply by the energy estimate and Sobolev inequality}. Similar continuum limit problems are considered on the infinite cubic lattice $h\mathbb{Z}^d$ \cite{Hong2019a} and on the finite periodic lattice $h\mathbb{Z}^d/2\pi\mathbb{Z}^d$ \cite{Hong2021}. However, in the latter compact case, owing to a technical difficulty (see \cite[Remark 1.4]{Hong2021}), we are currently unable to obtain the global-in-time convergence in three dimensions. A key observation in this study is that the expanding domain $\Omega$ is somewhat intermediate between bounded and unbounded domains. The challenge can be circumvented by a better dispersion from the domain expansion effect (see Remark \ref{Strichartz remark} below). Consequently, the physically important three-dimensional case is included. In fact, for numerical simulation, it is more desirable to consider a lattice domain whose size is independent of $h$ and to obtain an approximation bound in terms of independent parameters $h$ and $R$. However, we are currently unable to include such a model due to the same technical difficulty in the periodic case. We also leave it to future work.

To prove the main theorem, we follow the general strategy in our previous studies \cite{Hong2021, Hong2021a, Hong2019, Hong2019a}. First, by developing harmonic analysis tools on a discrete setting, we investigate the dispersive properties for the corresponding linear equation. For the NLS, they are stated in the form of Strichartz estimates \cite{Hong2021, Hong2019, Hong2019a}. For the Fermi-Pasta-Ulam system, the local smoothing and maximal function estimates are proved for the linear flow \cite{Hong2021a}. Then, we show that nonlinear solutions possess similar regularization properties as those of linear flows. Finally, using the properties of nonlinear solutions, we directly measure the difference between discrete and continuous solutions.

One crucial difference between our previous and current studies is that, unlike the infinite or the finite periodic lattice\footnote{For the definition of the Fourier transform on $h\mathbb{Z}^d$ (resp., that on the finite periodic lattice), we referred to Section \ref{sec:3} (resp., \cite{Hong2021}).}, the Fourier transform is not defined appropriately on the finite lattice $\Omega$ because it is not translation-invariant. Thus, as an alternative, we may utilize the orthonormal basis consisting of eigenfunctions of the discrete Laplacian with the zero-boundary condition, but then some favorable algebraic properties are limited and it causes several technical difficulties. For instance, in order to compare global solutions in different domains, we need to control their weighted norms. In the infinite case, the weighted norm can be estimated (Proposition \ref{prop:xvL2}), using the simple commutator identity (Lemma \ref{commutator relation hZ}). In contrast, it is a bit tricky to obtain an analogous result in the finite case (Proposition \ref{prop:xwL2}). Indeed, in this setting, multiplication by $x$ cannot be translated to a differential operator on the frequency side. Thus, we should find a suitable weight function $\varphi$ which behaves like $x$ but does not generate too complicated terms from the commutator (Lemma \ref{lem:commute}).

We also note that by the eigenfunctions, the core oscillatory sum can be obtained from the linear Schr\"odinger flow (see \eqref{kernel for linear propagator}), but it includes more terms because the Fourier convolution property $\widehat{fg}=\hat{f}*\hat{g}$ is no longer available. However, it turns out that these eigenfunctions has several similarities with the Fourier bases on a torus in many computational aspects. It leads us to obtain a desired bound for the sum which is consistent with both infinite and finite periodic cases \cite{Hong2021, Hong2019}.\\

Once the oscillatory sum is estimated, we obtain Strichartz estimates, which are the key inequalities on a finite lattice $\Omega$. For the statement, we define the $L^p(\Omega)$ norm by
\begin{equation}\label{L^p norm}
\|f\|_{L^p(\Omega)}:=\left\{\begin{aligned}&\left\{h^d\sum_{x\in \Omega}|f(x)|^p \right\}^{1/p}&&\textup{if }1\leq p<\infty,\\
&\sup_{x\in\Omega} |f(x)|&&\textup{if }p=\infty.
\end{aligned}\right.
\end{equation}
We say that $(q,r)$ is \textit{lattice-admissible} if $2\leq q,r\leq\infty$,
\begin{equation}\label{r-admissible}
\frac{3}{q}+\frac{d}{r}=\frac{d}{2} \;\; \textup{ and } \;\; (q,r,d)\neq (2,\infty,3).
\end{equation}

\begin{theorem}[Strichartz estimates]\label{thm:Stri}
Let $d=2, 3$, and let $K, R \in \N$ such that $K = \frac{\pi}{h}$ and $R \sim h^{-\alpha}$ with $\alpha \geq 0$ for some $h\in(0,1]$. 
For a lattice-admissible pair $(q,r)$, we have
\begin{equation}\label{Strichartz estimate 1}
\left\| e^{it\Delta_\Omega }w_0\right\|_{L_t^q([0,1]; L^r(\Omega))}\lesssim \|(1-\Delta_\Omega)^{\frac{s}{2}}w_0 \|_{L^2(\Omega)}, 
\end{equation}
where
\begin{equation}\label{eq:gamma}
\left\{\begin{aligned}
&s>\tfrac{1}{q} &&\textup{if }\alpha\geq1,\\
&s=\tfrac{2-\alpha}{q}&& \textup{if }0<\alpha<1,\\
&s>\tfrac{2}{q}&&\textup{if }\alpha=0.
\end{aligned}\right.
\end{equation}
\end{theorem}

\begin{remark}\label{Strichartz remark}
$(i)$ One of the key features of the discrete Schr\"odinger model is that it experiences a weaker dispersion than its continuous counterpart. This is because the Fourier symbol of the discrete Laplacian may have a degenerate Hessian. Therefore, the admissible equation \eqref{r-admissible} is different from that in the Euclidean case. Putting the derivative of order $\frac{1}{q}$ on the right-hand side cannot be avoided  (refer to the counterexample in \cite{Hong2019}).\\
$(ii)$ Compared to the infinite case \cite{Hong2019}, there is an additional loss $s-\frac{1}{q}$ in \eqref{Strichartz estimate 1}. Note that this additional loss decreases to $0^+$ as the speed of expansion increases (as $\alpha \uparrow 1$)\footnote{Even when $\alpha\geq 1$, $0^+$-loss appears when the Littlewood-Paley projections are summed up.} (see Table \ref{tab:table1} below). We do not expect optimality for this loss. Indeed, for simplicity and to avoid technical difficulty, we employ the argument of Vega \cite{V-92} who provided a simpler proof of Strichartz estimates on the torus $\mathbb{T}^d$. We expect that as in the work of Bourgain \cite{Bourgain1993}, the loss of regularity can be reduced by analyzing the hypersurfaces and their intersections associated with the discrete linear Schr\"odinger flow. \\
$(iii)$ To prove Theorem \ref{thm:main}, a slightly modified time-averaged $L^\infty(\Omega)$-bound (Proposition \ref{linear L^infty bound}) is employed; however, it is essentially a trivial modification (see Rermark \ref{L infinity remark1} and \ref{L infinity remark2}). \\
$(iv)$ For Theorem \ref{thm:main}, the loss of regularity in Strichartz estimates \eqref{Strichartz estimate 1} (or Proposition \ref{linear L^infty bound}) is sufficient. If the loss is reduced further, then the continuum limit can be proved in even higher dimensions and for a larger class of nonlinearities.
\end{remark}

\begin{table}[h!]
\centering
%\caption{Table template 4}
%\label{t4}

\begin{tabular}{ |m{0.01cm}|m{1.8cm}||m{2.5cm}|m{2.5cm}|m{2.5cm}|m{2.5cm}|m{0.01cm}| }
% \hline
% \multicolumn{7}{|c|}{Comparison with loss of derivatives upon domains} \\
 \hline\hline
 &&  \centering Periodic  lattice $\T_h^d$ &  \centering Expanding lattice $\Omega$, $0 < \alpha < 1$ &  \centering Expanding lattice $\Omega$, $\alpha \ge 1$ & \centering Infinite lattice $h\Z^d$ &\\
 \hline
 &  \centering Loss of derivatives   & \centering $\frac2q +$ & \centering $\frac{2-\alpha}{q}$ & \centering $\frac{1}{q} +$ & \centering $\frac1q$ \textup{\tiny(necessary)}&\\
% Aland Islands&   AX  & ALA   &248 &\\
% Albania &AL & ALB&  008 &\\
% Algeria    &DZ & DZA&  012 &\\
% American Samoa&   AS  & ASM&016 &\\
% Andorra& AD  & AND   &020 & \\
% Angola& AO  & AGO&024 & \\
\hline\hline
\end{tabular}
%
%\begin{tabular}{|c|c|c|c|c|}
%%\noalign{\smallskip}\noalign{\smallskip}
%\hline\hline
%\multirow{4}{*}{} & \multicolumn{1}{|c|}{Periodic lattice $\T_h^d$} & \multicolumn{1}{|c|}{Expanding lattice $\Omega$} & \multicolumn{1}{|c|}{Expanding lattice $\Omega$} & \multicolumn{1}{|c|}{Infinite lattice $h\Z^d$} \\
%%\cline{2-5}
%%      & subcolumn1-1  & subcolumn1-2 & subcolumn2-1 & subcolumn2-2 \\
%\hline\hline
%
%Loss of derivatives  & $\frac2q$ & $\frac{2-\alpha}{q}$ & $\frac{2-\alpha}{q}$ & $\frac1q$  \\
% 
%\hline\hline
%\end{tabular}
\vspace{1em}
\caption{Comparison of the regularity loss of $L_t^qL_x^r$ norm from Strichartz estimate under different settings: With the same gird size ($h$), but on different domains with size $R \sim h^{-\alpha}$, solutions lose $\frac2q+$, $\frac{2-\alpha}{q}$, $\frac1q+$ and $\frac1q$ smoothness in $L_t^qL_x^r$ Strichartz norm when $\alpha = 0$, $0 < \alpha <1$, $\alpha \ge 1$ and $\alpha = \infty$, respectively.}
\label{tab:table1}
\end{table}

\subsection{Reduction of the proof of Theorem \ref{thm:main}}\label{sec:outline}

The proof of the main theorem can be reduced employing the previously known result \cite{Hong2019a} on the intermediate model on the infinite lattice $h\Z^d=\{hm : m\in\mathbb{Z}^d\}$. We consider the NLS on an infinite lattice, 
\begin{equation}\label{eq:DNLS}
\left\{\begin{aligned}
i\partial_t v +\Delta_h v - |v|^2v&=0,\\
v(0) &= v_0,
\end{aligned}\right.
\end{equation}
where
$$v=v(t,x):\mathbb{R}\times h\mathbb{Z}^d\to\mathbb{C}$$
and the discrete Laplacian $\Delta_h$ is defined by
\[(\Delta_h v)(x) = \sum_{j=1}^d \frac{u(x+h\mathbf{e}_j) + u(x-h\mathbf{e}_j) - 2u(x)}{h^2}.\]
Given initial data $u_0\in H^1(\mathbb{R}^d)$, let $v(t)$ be the solution to the discrete NLS \eqref{eq:DNLS} with initial data with initial data $d_hu_0$ (see \eqref{discretization}). Then, it is shown in \cite{Hong2019a} that $\ell_hv(t)$ converges to the solution $u(t)$ to the NLS \eqref{eq:NLS} with initial data $u_0$. Therefore, for the proof of Theorem \ref{thm:main}, it is enough to prove the convergence from the infinite discrete NLS \eqref{eq:DNLS} to the finite one \eqref{eq:DRNLS} as $h \to 0$ as well as $R = h^{-\alpha} \to \infty$ (see Section \ref{sec:7} for more details). This reduction is illustrated in Figure \ref{Fig:Answer}.
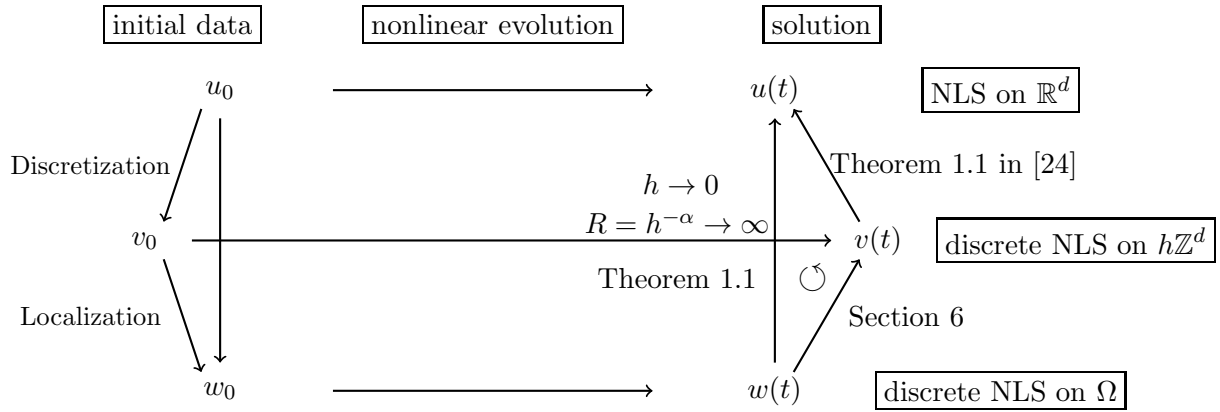
\begin{figure}[h!]
\begin{center}
\begin{tikzpicture}[scale=0.5]
\node at (-7.25,-6){$w_0$};
\draw[thick,->] (-4.25,-6) -- (4.25,-6);
\node at (-8.25,3.75){\boxed{\text{initial data}}};
\node at (-7.25,2){$u_0$};
\node at (-9.25,-2){$v_0$};
\draw[thick,->] (-4.25,2) -- (4.25,2);
\draw[thick,->] (-8,-2) -- (9,-2);
\draw[thick,->] (-7.75,1.5) -- (-8.75,-1.5);
\draw[thick,->] (-8.75,-2.5) -- (-7.75,-5.5);
\node at (0,3.75){\boxed{\text{nonlinear evolution}}};
\draw[thick,->] (-7.25,1.25) -- (-7.25,-5.25);
\node at (8.5,-3){\Large{$\circlearrowleft$}};
\node at (-10.7,0){{\small Discretization}};
\node at (-10.7,-4){{\small Localization}};
\node at (7.5,-6){$w(t)$};
\draw[thick,->] (8,-5.5) -- (9.75,-2.5);
\draw[thick,->] (9.75,-1.5) -- (8,1.5);
\node at (12.25,0){Theorem 1.1 in \cite{Hong2019a}};
\node at (13.5,-6){\boxed{\text{discrete NLS on } \Omega}};
\node at (10.25,-2){$v(t)$};
\node at (15.5,-2){\boxed{\text{discrete NLS on } h\Z^d}};
\node at (8.75,3.75){\boxed{\text{solution}}};
\node at (7.5,2){$u(t)$};
\node at (11,-4){Section \ref{sec:7}};
\node at (13.5,2){\boxed{\text{NLS on } \R^d}};
\draw[thick, <-] (7.5,1.25) -- (7.5,-5.25);
\node at (5,-0.5){$h \to 0$};
\node at (4.9,-1.5){$R = h^{-\alpha} \to \infty$};
\node at (4.9,-3){Theorem \ref{thm:main}};
\end{tikzpicture}
\end{center}
\caption{Convergence scheme from discrete NLS on $\Omega$ to NLS on $\R^d$.}\label{Fig:Answer}
\end{figure}
For the rest, we follows the approach in the previous work \cite{Hong2019a}. To be precise, we establish Strichartz estimates under the discrete and bounded setting that validate the uniform $L^{\infty}(\Omega)$ bounds for linear and nonlinear solutions. We, however, further require decay (in $x$ variable) properties of both solutions to \eqref{eq:DNLS} and \eqref{eq:DRNLS} at any given time; i.e., both $xv(t)$ and $xw(t)$ in $L^2$ do not blow up in finite time. Note that such a property is well-known for the NLS solutions in $\R^d$.
%The proof of Theorem 1.1 in \cite{Hong2019a} is largely based on techniques in harmonic analysis and theories of dispersive PDEs. Strichartz estimates under the discrete setting play a key role in this study and   enable us to establish the uniform $L^\infty(h\Z^d)$ bounds for linear and nonlinear solutions. Analogously, we establish Strichartz estimates under the discrete and bounded setting that validate the uniform $L^{\infty}(\Omega)$ bounds for linear and nonlinear solutions. In contrast to \cite{Hong2019a}, we further require decay (in $x$ variable) properties of both solutions to \eqref{eq:DNLS} and \eqref{eq:DRNLS} at any given time; i.e., both $xv(t)$ and $xw(t)$ in $L^2$ do not blow up in finite time. Note that such a fact for the NLS solutions in $\R^d$ is well-known. In fact, $xu(t)$ in $L^2$ is bounded in time, and it immediately follows from the conservation laws of energy and pseudo-conformal energy. Moreover, it guarantees local smoothing effect from localized data. In later sections, we briefly introduce essential tools for Theorem 1.1 in \cite{Hong2019a} and provide corresponding properties under the discrete and bounded setting.

\subsection{Organization of the paper}
In Section \ref{sec:3}, we summarize continuum limit of discrete NLS on the infinite domain ($h\Z^d$). In Section \ref{sec:4}, spectral properties for discrete Laplacian and fundamental analyses on discrete setting are introduced. In Section \ref{sec:5}, the Strichartz estimates are proved, and by using this, uniform bounds for solutions are established in Section \ref{sec:6}. In Section \ref{sec:7}, we prove Theorem \ref{thm:main}. The proof of spectral properties is given in Appendix \ref{proof of spectral properties}, and in Appendix \ref{sec:small amplitude limit}, the mathematical equivalence with Theorem \ref{thm:main} (small amplitude limit) is introduced.

\subsection{Acknowledgement}
This research of the first author was supported by Basic Science Research Program through the National Research Foundation of Korea (NRF) funded by the Ministry of Science and ICT (NRF-2020R1A2C4002615).This work of the second author was supported by the National Research Foundation of Korea (NRF) grant funded by the Korea government (MSIT) (No. 2020R1F1A1A0106876811). C. Yang was supported by the National Research Foundation of Korea(NRF) grant funded by the Korea government(MSIT) (No. 2021R1C1C1005700).

\section{Nonlinear Schr\"odinger equation on an infinite cubic lattice}\label{sec:3}

We summarize the setup and some results in \cite{Hong2019, Hong2019a} for the NLS on an infinite lattice, which will be used for the proof of the main theorem. Moreover, we obtain a new simple weighted norm bound for nonlinear solutions (Proposition \ref{prop:xvL2}). Indeed, a large portion of this article is devoted to reproducing analogous results in a finite setting. Hence, this section may be considered as a guide for later sections. 

\subsection{Function spaces on an infinite lattice}
For $f:h\mathbb{Z}^d\to\mathbb{C}$, we define the Lebesgue norm by
$$\|f\|_{L^p(h\mathbb{Z}^d)}:=\left\{\begin{aligned}&\left\{h^d\sum_{x\in h\mathbb{Z}^d}|f(x)|^p \right\}^{1/p}&&\textup{if }1\leq p<\infty,\\
&\sup_{x\in h\mathbb{Z}^d} |f(x)|&&\textup{if }p=\infty,
\end{aligned}\right.$$
the Sobolev norm by
$$\|f\|_{H^s(h\mathbb{Z}^d)}:=\|(1-\Delta_h)^{\frac{s}{2}}f\|_{L^2(h\mathbb{R}^d)}$$
and the weighted Sobolev norm by
\begin{equation}\label{eq:weighted Sobolev}
\|f\|_{H^{1,1}(h\mathbb{Z}^d)}:=\left\{\|\sqrt{1-\Delta_h}f\|_{L^2(h\mathbb{R}^d)}^2+\|xf\|_{L^2(h\mathbb{Z}^d)}^2\right\}^{1/2}.
\end{equation}
The function spaces $L^p(h\mathbb{Z}^d)$, $H^s(h\mathbb{Z}^d)$ and $H^{1,1}(h\mathbb{Z}^d)$ are respectively the sets of functions equipped with the corresponding norms. In particular, $L^2(h\mathbb{Z}^d)$ is a Hilbert space with the inner product $h^d\sum_{x\in h\mathbb{Z}^d} f(x)\overline{g(x)}$.

We define the right difference gradient by
\begin{equation}\label{right difference gradient}
(\nabla_hf)(x)=\sum_{j=1}^d\frac{f(x+h\mathbf{e}_j)-f(x)}{h} \mathbf{e}_j.
\end{equation}
%where $\{\mathbf{e}_1,\cdots, \mathbf{e}_d\}$ is the standard basis for $\mathbb{R}^d$. 
Then, its adjoint is the left difference gradient $(\nabla_h^*f)(x)=\sum_{j=1}^d\frac{f(x)-f(x-h\mathbf{e}_j)}{h}\mathbf{e}_j$, and $\nabla_h\cdot\nabla_h^*=\nabla_h^*\cdot\nabla_h=\Delta_h$. Thus, we have
\begin{equation}\label{right derivative L^2}
\|\sqrt{-\Delta_h}f\|_{L^2(h\mathbb{Z}^d)}=\|\nabla_h f\|_{L^2(h\mathbb{Z}^d)},\quad \|f\|_{H^1(h\mathbb{Z}^d)}^2=\|f\|_{L^2(h\mathbb{Z}^d)}^2+\|\nabla_h f\|_{L^2(h\mathbb{Z}^d)}^2.
\end{equation}

\subsection{Fourier transform on an infinite lattice}
For a function $f : h\Z^d \to \C$, its Fourier transform is defined by
\[\hat{f}(\xi)= (\mathcal F_h f)(\xi) := h^d \sum_{x \in h\Z^d} e^{-ix \cdot \xi} f(x), \quad \xi \in \T_h^d,\]
where  $\T_h^d = \frac{2\pi}{h}[-\frac12, \frac12)^d$ is a periodic box, and its inversion formula is given by 
\[\check{f}(x) = (\mathcal F_h^{-1}f)(x) := \frac{1}{(2\pi)^d} \int_{\T_h^d} e^{ix \cdot \xi} f(\xi) \; d\xi\]
for functions on $\T_h^d$. By definition, the Fourier transform on $h\mathbb{Z}^d$ is essentially the inverse Fourier transform on the periodic box $\T_h^d$. Hence, many of basic properties immediately follow from what are known for the Fourier series. For instance, the Parseval identity and the Plancherel theorem
$$h^d \sum_{x \in h\Z^d}f(x) \overline{g(x)}=\frac{1}{(2\pi)^d}\int_{\T_h^d}\hat{f}(\xi) \overline{\hat{g}(\xi)}d\xi, \qquad\norm{f}_{L^2(h\Z^d)} = \frac{1}{(2\pi)^{d/2}}\|\hat{f}\|_{L^2(\T_h^d)}$$
hold. We refer to \cite{Hong2019} for more properties and their proofs.

A direct calculation shows that the discrete Laplacian $-\Delta_h$ is a Fourier multiplier, that is, $\mathcal{F}_h(-\Delta_h f)(\xi)=\mathcal{P}_h(\xi)(\mathcal{F}_hf)(\xi)$, of symbol 
\begin{equation}\label{P_h}
\mathcal P_h(\xi):=\frac{2}{h^2}\sum_{j=1}^2 1-\cos(h\xi_j)=\sum_{j=1}^d \frac{4}{h^2}\sin^2\Big(\frac{h\xi_j}{2}\Big).
\end{equation}
Hence, $\|f\|_{H^s(h\mathbb{Z}^d)}=\frac{1}{(2\pi)^{d/2}}\|(1+\mathcal{P}_h(\xi))^{\frac{s}{2}}\hat{f}\|_{L^2(\mathbb{T}_h^d)}$.

For the Littlewood-Paley theory on $h\mathbb{Z}^d$, we choose an axisymmetric smooth bump function $\tilde{\psi}$ such that $\tilde{\psi}(\xi) \equiv 1$ on $[-1,1]^d$ but $\tilde{\psi}(\xi) \equiv 0$ on $\R^d \setminus [-2,2]^d$, and let $\psi = \tilde{\psi}(\cdot) - \tilde{\psi}(2\cdot)$. Then, for a dyadic number $N \in 2^{\Z}$,
\[\psi_N(\xi) := \psi\Big( \frac{h\xi}{\pi N}\Big)\]
satisfies $\mbox{supp } \psi_N \subset [-\frac{4\pi N}{h}, \frac{4\pi N}{h}]^d \setminus [-\frac{\pi N}{h},\frac{\pi N}{h}]^d$ and $\sum_{N \le 1} \psi_N \equiv 1$ on $\mathbb{T}_h^d$. We now define the Littlewood-Paley projection operator $P_N$ by 
\[P_N f = \mathcal{F}_h^{-1}\big(\psi_N(\cdot)(\mathcal{F}_hf)\big).\]
By construction, we have
$$\|f\|_{L^2(h\mathbb{Z}^d)}^2\sim\sum_{N\leq 1}\|P_Nf\|_{L^2}^2.$$

\subsection{Strichartz estimates on an infinite lattice}
On an infinite lattice, the Schr\"odinger flow from each dyadic piece enjoys some decay property.

\begin{proposition}[Dispersive estimate {\cite[Proposition 5.1]{Hong2019}}]\label{dispersion hZd}
Let $h \in (0,1]$. For any dyadic number $N\in 2^{\mathbb{Z}}$ with $N\leq 1$, we have
\[
\| e^{it \Delta_h} P_{N}v_0\|_{L^\infty(h\Z^d)} \ls  \left(\frac{N}{|t|h}\right)^{\frac d3}\|v_0 \|_{L^1(h\Z^d)}\quad\textup{for all }t\in\mathbb{R}.
\]
\end{proposition}

By the argument in Keel-Tao \cite{KT1998}, Strichartz estimates are deduced from Proposition \ref{dispersion hZd}.
\begin{theorem}[Strichartz estimates {\cite[Theorem 1.3]{Hong2019}}]\label{hZd Strichartz}
Let $h \in (0,1]$. For any admissible pair $(q,r)$ (see \eqref{r-admissible}), we have
\[ \| e^{it\Delta_h}v_0\|_{L_t^q(\R; L^r(h\Z^d))}\lesssim \|(-\Delta_h)^{\frac{1}{2q}}v_0 \|_{L^2(h\Z^d)}. \]
\end{theorem}

\subsection{NLS on an infinite lattice}

We consider the discrete cubic defocusing NLS \eqref{eq:DNLS} on an infinite lattice. For fixed $h\in(0,1]$, its initial-value problem is easily shown to be globally well-posed in $L^2(h\Z^d)$ (see \cite[Proposition 6.1]{Hong2019}). Indeed, local well-posedness is obvious, since the difference operator $\Delta_h$ is bounded on $L^2(h\mathbb{Z}^d)$. Moreover, solutions preserve the mass and the energy
$$M(v) = \|v\|_{L^2(h\Z^d)}^2\quad\textup{and}\quad E(v) = \frac12 \|\nabla_hv\|_{L^2(h\Z^d)}^2 + \frac14\|v\|_{L^4(h\Z^d)}^4.$$
These two conservation laws assure global well-posedness.

Beside existence, Strichartz estimates (Theorem \ref{hZd Strichartz}) yield a time-averaged $L^\infty$-bound for nonlinear solutions.

\begin{proposition}[Uniform time-averaged $L^\infty$-bound {\cite[Proposition 4.2]{Hong2019a}}]\label{uniform v}
Let $d=2, 3$. For $h \in (0,1]$, we assume that $v_0\in H^1(h\mathbb{Z}^d)$. Then, for any sufficiently small $\delta>0$, the global solution $v(t)$ to NLS  \eqref{eq:DNLS} with initial data $v_0$ satisfies 
\[\|v\|_{L_t^{\frac{4}{d-2+\delta}}([-T,T]; L^\infty(h\Z^d))} \lesssim \langle T\rangle^{\frac{d-2+\delta}{4}}\quad\textup{for all }T>0,\]
where the implicit constant depends only on $\|v_0\|_{H^1(h\mathbb{Z}^d)}$ and $\delta>0$.
\end{proposition}

Proposition \ref{uniform v} is crucially employed in \cite{Hong2019a} to establish the continuum limit to NLS \eqref{eq:NLS}. It will also be used in the proof of our main theorem.

\begin{theorem}[Continuum limit from NLS \eqref{eq:DNLS} {\cite[Theorem 1.1]{Hong2019a}}]\label{YH theorem}
Let $d=2,3$. Suppose that $u_0\in H^1(\mathbb{R}^d)$, and we let $u(t)$ be the global solution to NLS \eqref{eq:NLS} with initial data $u_0$ and let $v(t)$ be the global solution to NLS \eqref{eq:DNLS} with initial data $d_hu_0$. Then, there exists $C, c>0$, depending only on $\|u_0\|_{H^1(\mathbb{R}^d)}$, such that 
$$\|u(t)-\ell_hv(t)\|_{L^2(\mathbb{R}^d)}\leq C\sqrt{h}e^{ct}.$$
\end{theorem}

In addition, we need a weighted norm bound for comparison with solutions on a finite lattice. 

\begin{proposition}[Uniform weighted $L^2$-bound]\label{prop:xvL2}
Let $d=2, 3$. For $h \in (0,1]$, we assume that $v_0\in H^{1,1}(h\mathbb{Z}^d)$, and let $v(t)$ be the global solution to NLS \eqref{eq:DNLS} with initial data $v_0$. Then, there exists $c>0$, depending only on $\|u_0\|_{H^{1,1}(h\mathbb{Z}^d)}$, such that
$$\|xv(t)\|_{L^2(h\Z^d)}\lesssim \|u_0\|_{H^{1,1}(h\mathbb{Z}^d)}e^{ct}\quad\textup{ for all }t\in\mathbb{R}.$$
\end{proposition}

For the proof, we make use of the following commutator relation which corresponds to the well-known identity $[x,e^{it\Delta}]=-2it\nabla e^{it\Delta}$ on the continuous domain $\mathbb{R}^d$.

\begin{lemma}[Commutator identity]\label{commutator relation hZ}
$$[x,e^{it\Delta_h}]=-2it\tilde{\nabla}_he^{it\Delta_h},$$
where $\tilde{\nabla}_h$ is the Fourier multiplier with symbol $\frac{i}{2}\nabla\mathcal{P}_h(\xi)$.
\end{lemma}

\begin{proof}
By integration by parts, one can show that $\mathcal F_hx=i\nabla_\xi\mathcal{F}_h$. Hence, we have
$$\begin{aligned}
\mathcal{F}_h [ x,e^{it\Delta_h}]&=\mathcal{F}_hxe^{it\Delta_h}- \mathcal{F}_he^{it\Delta_h}x=i\nabla_{\xi}e^{-it\mathcal{P}_h(\xi)}-ie^{-it\mathcal{P}_h(\xi)}\nabla_{\xi}= t(\nabla\mathcal{P}_h)(\xi)e^{-it\mathcal{P}_h(\xi)}.
\end{aligned}$$
Inverting the Fourier transform, the lemma follows.
\end{proof}

\begin{proof}[Proof of Proposition \ref{prop:xvL2}]
By the Duhamel formula and Lemma \ref{commutator relation hZ}, we write 
\[\begin{aligned}
xv(t) =&~{} xe^{it\Delta_h} v_0 - i\int_0^t xe^{i(t-s)\Delta_h} (|v|^2v)(s) \; ds\\
=&~{} e^{it\Delta_h} (x-2it\tilde{\nabla}_h)v_0 - i\int_0^t e^{i(t-s)\Delta_h} (x-2i(t-s)\tilde{\nabla}_h)(|v|^2v)(s) \; ds.
\end{aligned}\]
Note that since $|\frac{i}{2}\nabla\mathcal{P}_h(\xi)|=|\sum_{j=1}^d\frac{2}{h}\sin(\frac{h\xi_j}{2})\cos(\frac{h\xi_j}{2})\mathbf{e}_j|\leq\sqrt{\mathcal{P}_h(\xi)}$, by the Plancherel theorem and \eqref{right derivative L^2}, we have 
\[\|\tilde{\nabla}_h f\|_{L^2(h\Z^d)} \le \|\sqrt{-\Delta_h}f\|_{L^2(h\Z^d)}=\|\nabla_h f\|_{L^2(h\mathbb{Z}^d)}.\]
Thus, it follows that 
\[\begin{aligned}
\|xv(t)\|_{L^2(h\mathbb{Z}^d)} &\leq~{} \|x v_0\|_{L^2(h\mathbb{Z}^d)}+2|t|\|\nabla_h v_0\|_{L^2(h\mathbb{Z}^d)} \\
&\quad+\int_0^t \left\{\|x (|v|^2v)(s)\|_{L^2(h\mathbb{Z}^d)}+2|t-s|\|\nabla_h(|v|^2v)(s)\|_{L^2(h\mathbb{Z}^d)}\right\}  ds
\end{aligned}\]
For the nonlinear term, we observe that 
\begin{equation}\label{eq:product}
\begin{aligned}
\frac{|v|^2v(x+h\mathbf{e}_j)-|v|^2v(x)}{h}=&~{}\frac{v(x+h\mathbf{e}_j)-v(x)}{h}|v(x+h\mathbf{e}_j)|^2\\
&~{}+v(x)\frac{\bar{v}(x+h\mathbf{e}_j)-\bar{v}(x)}{h}v(x+h\mathbf{e}_j)\\
&~{}+|v(x)|^2\frac{v(x+h\mathbf{e}_j)-v(x)}{h},
\end{aligned}
\end{equation} 
and thus,
$$\|\nabla_h(|v|^2v)\|_{L^2(h\Z^d)}\lesssim \|v\|_{L^\infty(h\mathbb{Z}^d)}^2\|\nabla_hv\|_{L^2(h\Z^d)}.$$
Inserting this bound, we obtain 
$$\begin{aligned}
\|xv(t)\|_{L^2(h\Z^d)}&\lesssim ~{} \|xv_0\|_{L^2(h\Z^d)} +2|t|\|\nabla_hv_0\|_{L^2(h\Z^d)} \\
&\quad+ \int_0^t \|v(s)\|_{L^\infty(h\mathbb{Z}^d)}^2\left\{\|xv(s)\|_{L^2(h\Z^d)}+2|t-s|\|\nabla_hv(s)\|_{L^2(h\Z^d)}\right\} ds\\
&\leq (1+2|t|)\|v_0\|_{H^{1,1}(h\Z^d)}+4|t|E(v_0) |t|^{1-\frac{2}{q}}\|v(s)\|_{L_s^q([0,t]L^\infty(h\mathbb{Z}^d))}^2\\
&\quad+ \int_0^t \|v(s)\|_{L^\infty(h\mathbb{Z}^d)}\|xv(s)\|_{L^2(h\Z^d)}ds,
\end{aligned}$$
where we used the energy conservation law in the last inequality. Therefore, by Gr\"onwall's inequality and Proposition \ref{uniform v}, we conclude that $\|xv(t)\|_{L^2(h\Z^d)}$ satisfies the desired exponential bound.
\end{proof}

\section{Preliminaries on a finite cubic lattice}\label{sec:4}

From now on, if nothing is mentioned, we fix $K, R \in \N$ such that $K = \frac{\pi}{h}$ and $R \sim h^{-\alpha}$ for some $\alpha > 0$ and small $h\in (0,1]$, and let $\Omega=\Omega_{h,\pi R}$ be a finite cubic lattice. The purpose of this section is to introduce basic function spaces, the discrete Laplacian, and an orthonormal basis on the finite lattice $\Omega$, which are consistent with the infinite case in Section \ref{sec:3}. 

\subsection{Lebesgue spaces on a finite lattice}
For $1 \le p \le \infty$, the $L^p(\Omega)$-space is defined as the collection of all functions $f: \Omega\to\mathbb{C}$ with zero boundary $f|_{\partial \Omega} = 0$, equipped with the norm $\|\cdot\|_{L^p(\Omega)}$ (see \eqref{L^p norm}).
Here, we denote the $L^2(\Omega)$-inner product by
$$\langle f,g\rangle:=\sum_{x\in\Omega}f(x)\overline{g(x)}.$$
Since $\|\mathcal{E}f\|_{L^p(h\mathbb{Z}^d)}=\|f\|_{L^p(\Omega)}$ where the extension operator $\mathcal{E}$ is given by \eqref{eq: extension},  all basic inequalities such as H\"older's inequality immediately follow from those on the infinite lattice $h\mathbb{Z}^d$ (see \cite[Section 2.1]{Hong2019}).

\subsection{Discrete Laplacian, an orthonormal basis and Sobolev spaces on a finite lattice}
We recall the definition of the discrete Laplacian $\Delta_\Omega$ from \eqref{eq:LaplaceOmega}. Indeed, it is self-adjoint on $L^2(\Omega)$. To characterize its spectral properties, we introduce the frequency domain
$$\Omega^*:=\Big\{\xi\in \tfrac{1}{2R}\mathbb{Z}^d: 0 < \xi_j<\tfrac{\pi}{h}\Big\}=\Big\{\xi=\tfrac{m}{2R}: m_j=1, 2, ...,2KR-1 \Big\},$$
and we let 
\begin{equation}\label{eq: eigenfunction}
\begin{aligned}
e(x,\xi)&=\frac{1}{(\pi R)^{d/2}}\prod_{j=1}^d \sin((x_j+\pi R)\xi_j),\quad \xi\in\Omega^*\\
&=\frac{1}{(\pi R)^{d/2}}\prod_{j=1}^d \sin(\tfrac{m_j}{2R}x_j+\tfrac{\pi}{2}m_j),\quad m_j=0,1, ..., 2KR-1.
\end{aligned}
\end{equation}

%{\color{red} $e(x,\xi) = 0$ if at least one component of $\xi$ is zero.}

\begin{lemma}[Spectral properties for $-\Delta_\Omega $]\label{spectral properties}
Given $\xi\in \Omega^*$, $e(x,\xi)$ solves the eigenvalue equation 
\begin{equation}\label{eq:eigenvalue}
-\Delta_\Omega  e(x,\xi)=\mathcal{P}_h(\xi)e(x,\xi),
\end{equation}
where $\mathcal{P}_h(\xi)$ is given by \eqref{P_h}. Moreover, their collection $\{e(x,\xi)\}_{\xi\in\Omega^*}$ is an orthonormal basis for $L^2(\Omega)$.
\end{lemma}

\begin{proof}
The proof will be given in Appendix \ref{proof of spectral properties}.
\end{proof}

\begin{remark}\label{rem:eboundary}
An eigenfunction $e(x,\xi)$ is nothing but a product of sine and cosine functions with zero boundary. Indeed, we have
$$\sin((x_j+\pi R)\xi_j)=\sin(\tfrac{m_j}{2R}x_j+\tfrac{\pi}{2}m_j)=\left\{\begin{aligned}&\pm\sin(\tfrac{m_j}{2R}x_j)&&\textup{when }m_j\textup{ is even,}\\
&\pm\cos(\tfrac{m_j}{2R}x_j)&&\textup{when }m_j\textup{ is odd.}
\end{aligned}\right.$$
Just for notational convenience, $e(x,\xi)$ is expressed as a product of only sine functions.
\end{remark}

By Lemma \ref{spectral properties}, we may write
$$f(x)=\sum_{\xi\in\Omega^*}\langle f, e(\cdot,\xi)\rangle e(x,\xi).$$
Here, the function $\langle f, e(\cdot,\xi)\rangle: \Omega^*\to\mathbb{C}$ can be considered as a natural substitute for the Fourier transform of $f$.

For $s\in\mathbb{R}$, the Sobolev space $H^s(\Omega)$ is defined as the Hilbert space with the norm
$$\norm{f}_{H^s(\Omega)}:=\|(1-\Delta_\Omega )^{\frac{s}{2}}f\|_{L^2(\Omega)}\sim \left\{\sum_{\xi\in\Omega^*} \left(1+\mathcal{P}(\xi)\right)^{s}|\langle f, e(\cdot,\xi)\rangle|^2\right\}^{1/2}.$$
Note particularly that 
\[\norm{f}_{H^1(\Omega)}^2=\|f\|_{L^2(\Omega)}^2 + \|(-\Delta_\Omega )^{\frac{1}{2}}f\|_{L^2(\Omega)}^2.\]
Moreover, similarly as \eqref{right difference gradient}, we define the right difference gradient on $\Omega$ by
\[(\nabla_{\Omega}f)(x)=\begin{cases}\sum_{j=1}^d\frac{f(x+h\mathbf{e}_j)-f(x)}{h} \mathbf{e}_j, \quad &x \in \Omega \setminus \partial\Omega,\\ 0, \quad &x \in \partial \Omega.\end{cases}\]
where $\{\mathbf{e}_1,\cdots, \mathbf{e}_d\}$ is the standard basis for $\mathbb{R}^d$. Then, its adjoint is the left difference gradient on $\Omega$ defined by
\[(\nabla_{\Omega}^*f)(x)=\begin{cases}\sum_{j=1}^d\frac{f(x-h\mathbf{e}_j)-f(x)}{h} \mathbf{e}_j, \quad &x \in \Omega \setminus \partial\Omega,\\ 0, \quad &x \in \partial \Omega.\end{cases}\]
Thus, we have
\begin{equation}\label{adjoint L^2}
\|\sqrt{-\Delta_{\Omega}}f\|_{L^2(\Omega)}=\|\nabla_{\Omega} f\|_{L^2(\Omega)},\quad \|f\|_{H^1(\Omega)}^2=\|f\|_{L^2(\Omega)}^2+\|\nabla_{\Omega} f\|_{L^2(\Omega)}^2.
\end{equation}

\subsection{Dyadic decompositions on a finite lattice}\label{sec: Littlewood-Paley}
Let
$$N_*=2^{\ell_*}\quad\textup{with}\quad\ell_*=\lceil\log_2(\tfrac{h}{\pi})\rceil -1,$$
where $\lceil a\rceil$ denotes the smallest integer greater than or equal to $a$. For a dyadic number $N\in 2^{\mathbb{Z}}$ such that $N_*\leq N\leq 1$, we define the frequency projection operator $P_N$ by
\begin{equation}\label{LP projection}
(P_N f)(x) := \left\{\begin{aligned}
&\sum_{\frac{\pi N}{2h}<\max|\xi_j|\le \frac {\pi N}{h}}
\langle f, e(\cdot,\xi)\rangle e(x,\xi)&&\textup{if}\quad 2N_*\le N\leq 1,\\
&0 &&\textup{if}\quad N=N_*.
\end{aligned}\right.
\end{equation}
Indeed, in this case, smooth truncation is not necessary, because the frequency domain $\Omega^*$ is discrete. Note that by Lemma \ref{spectral properties},
\begin{equation}\label{eq:Hs}
\norm{f}_{H^s(\Omega)}^2\sim \sum_{N_* \le N \le 1} \left(1+(\tfrac{N}{h})^2\right)^{s}\|P_N f\|_{L^2(\Omega)}^2.
\end{equation}
Moreover, the projection $P_N$ satisfies the following natural boundedness property.

\begin{lemma}[Bernstein inequality]\label{Bernstein inequality}
Let $q\geq 2\geq p\geq 1$. Then, for $N\in 2^{\mathbb{Z}}$ with $N_*\leq N\leq 1$, we have
$$\|P_Nf\|_{L^q(\Omega)}\lesssim \left(\frac{N}{h}\right)^{d(\frac{1}{p}-\frac{1}{q})}\|f\|_{L^p(\Omega)}.$$
\end{lemma}

\begin{proof}
It is obvious that $\|P_Nf\|_{L^2(\Omega)}\leq\|f\|_{L^2(\Omega)}$ and
$$\|P_Nf\|_{L^\infty(\Omega)}\leq \sum_{\frac{\pi N}{2h}<\max|\xi_j|\le \frac {\pi N}{h}}
\|f\|_{L^1(\Omega)} \|e(\cdot,\xi)\|_{L^\infty(\Omega)}^2\lesssim \left(\frac{N}{h}\right)^{d}\|f\|_{L^1(\Omega)}.$$
Thus, the lemma follows from the standard $TT^*$ argument and real interpolation.
\end{proof}

%{\color{red}No advantage of having large $R$?}

\begin{remark}\label{Bernstein remark}
$(i)$ The exponents in Lemma \ref{Bernstein inequality} are restricted to the case $q\geq 2\geq p$. Indeed, in order to include the full range $q\geq p$, a more detailed information about the kernel of the projection $P_N$ is needed, but it seems technical. We here do not attempt to include the full range, because Lemma \ref{Bernstein inequality} is sufficient to prove the time-averaged $L^\infty(\Omega)$ estimates (Proposition \ref{linear L^infty bound}).\\
$(ii)$ By the Bernstein inequality, uniform-in-$h$ Sobolev and Gagliardo-Nirenberg inequalities are derived on a finite lattice, but again the exponent ranges are restricted at this moment.\\
$(iii)$ As seen in \eqref{eq: eigenfunction}, one can obtain further smoothing effect from $L^{\infty}(\Omega)$ norm of eigenvalues $e(x,\xi)$, i.e.,
$$\|P_Nf\|_{L^\infty(\Omega)} \lesssim \left(\frac{N}{h}\right)^{d(1-\alpha)}\|f\|_{L^1(\Omega)},$$
which guarantees 
$$\|P_Nf\|_{L^q(\Omega)}\lesssim \left(\frac{N}{h}\right)^{d(1-\alpha)(\frac{1}{p}-\frac{1}{q})}\|f\|_{L^p(\Omega)}.$$
However, this additional regularity gain does not play any significant role in the analysis for continuum limit, so we do not attempt to investigate the optimality in a sense of regularity.
\end{remark}

\section{Strichartz estimates on a finite cubic lattice: Proof of Theorem \ref{thm:Stri}}\label{sec:5}

In this section, we establish Strichartz estimates for the linear Schr\"odinger flow on a finite cubic lattice. As mentioned in the introduction, we here do not attempt to obtain an optimal result. Instead, allowing some additional loss of regularity, we employ the argument of Vega \cite{V-92}, because the proof is much simpler but the outcome is good enough to deal with the physically relevant 3d cubic equation.

As a first step, we reduce the proof of Strichartz estimates to that of the following ``short-time" dispersive estimate.

\begin{proposition}[Short-time dispersive estimate]\label{dispersion estimate}
For any dyadic number $N\in 2^{\mathbb{Z}}$ with $N_*:=2^{\lceil\log_2(\tfrac{h}{\pi})\rceil-1}\leq N\leq 1$,  %there exists $c>0$ such that 
if $|t|\le \frac{Rh}{2N}$, then
\begin{equation}\label{Linftybound:Dyadic}
\| e^{it \Delta_\Omega  } P_{\leq N}w_0\|_{L^\infty(\Omega)} \ls  \left(\frac{N}{|t|h}\right)^{\frac d3}\|w_0 \|_{L^1(\Omega)},
\end{equation}
where
$$(P_{\leq N}w_0)(x):=\frac{1}{(2\pi)^d}\sum_{\max \xi_j \leq \frac{\pi N}{h}}\langle w_0, e(\cdot,\xi)\rangle e(x,\xi).$$
\end{proposition}

\begin{remark}
For comparison, we recall from \cite[Theorem 1.5]{Hong2021} that on a periodic lattice with fixed size, a similar dispersion estimate is obtained for the linear Schr\"odinger flow on a time interval of size $\sim\frac{N}{h}$. An additional loss of regularity in Strichartz estimates is imposed when $O(\frac{N}{h})$ many intervals of size $O(\frac{h}{N})$ are summed up to an interval of length $1$. Proposition \ref{dispersion estimate} shows that on a larger domain, the dispersion estimate may hold longer. Indeed, in our application, $R$ is chosen to be comparable with $h^{-\alpha}$ for some $\alpha>0$, thus $\frac{Rh}{N}\gg\frac{h}{N}$. This is the place where regularity loss in Theorem \ref{thm:Stri} is reduced in Strichartz estimates compared to those on a periodic lattice.
\end{remark}

\begin{proof}[Proof of Theorem \ref{thm:Stri}, assuming Proposition \ref{dispersion estimate}]
When $0 \leq \alpha < 1$, we proceed as in the proof of \cite[Theorem 1.5]{Hong2021}. For convenience, we omit the spatial domain in the norm $\|\cdot\|_{L^r}=\|\cdot\|_{L^r(\Omega)}$. To begin with, by Proposition  \ref{dispersion estimate} and the standard interpolation argument \cite{KT1998}, we derive Strichartz estimates on the interval $[0,\frac{h^{1-\alpha}}{2N}]$, 
$$\begin{aligned}
\| e^{it \Delta_\Omega  } P_{N}w_0\|_{L_t^q([0,\frac{h^{1-\alpha}}{2N}];L^r)}&=\| e^{it \Delta_\Omega  } P_{\leq N}(P_Nw_0)\|_{L_t^q([0,\frac{h^{1-\alpha}}{2N}];L^r)}\\
&\ls  \left(\frac{N}{h}\right)^{\frac 1q}\|P_Nw_0 \|_{L^2}.
\end{aligned}$$
Moreover, by the temporal translation invariance, the time interval can be replaced by $I_j=[\frac{h^{1-\alpha}}{2N}(j-1),\frac{h^{1-\alpha}}{2N}j]\cap[0,1]$. Therefore, we have
\begin{equation}\label{eq:timespliting}\begin{aligned}
\|e^{it \Delta_\Omega  } P_N w_0\|_{L_t^q([0,1];L^r)}^{ q}&\leq~{}\sum_{j=1}^{\lceil\frac{2N}{h^{1-\alpha}}\rceil}\| e^{it \Delta_\Omega  } P_N w_0\|_{L_t^q(I_j;L^r)}^{q}\lesssim\sum_{n=1}^{\lceil\frac{2N}{h^{1-\alpha}}\rceil}\frac{N}{h}\|P_N w_0 \|_{L^2}^{q}\\
&\leq N^{\alpha} \left(\frac{N}{h}\right)^{2-\alpha}\|P_N w_0 \|_{L^2(\Omega)}^{q}\lesssim N^{\alpha}\|w_0 \|_{H^{\frac{2-\alpha}{q}}}^{q}.
\end{aligned}
\end{equation}
Finally, we collect all frequency pieces to obtain\footnote{An additional $\epsilon$-regularity loss arises in the summation over $N$ when $\alpha = 0$, see \cite[Theorem 1.5]{Hong2021}.}
$$\| e^{it \Delta_\Omega  }w_0\|_{L_t^q([0,1];L^r)}\leq\sum_{N=N_*}^1\|e^{it \Delta_\Omega  } P_N w_0\|_{L_t^q([0,1];L^r)}\ls\sum_{N=N_*}^1  N^{\frac{\alpha}{q}}\|w_0 \|_{H^{\frac{2-\alpha}{q}}}\lesssim\|w_0 \|_{H^{\frac{2-\alpha}{q}}}.$$

When $\alpha \ge 1$, the dispersive estimate \eqref{Linftybound:Dyadic} holds on the time interval $[0,1]$, independently of the choice of $h$ and $N$, thus we may skip the step \eqref{eq:timespliting}. However, $\epsilon$-loss is required when frequency pieces are summed up.
\end{proof}

It remains to show Proposition \ref{dispersion estimate}. We prove the proposition transferring the oscillatory summation into the corresponding oscillatory integral by the following lemma.

\begin{lemma}[Zygmund $\textup{\cite[Chapter V, Lemma 4.4]{Zygmund}}$]\label{Lem:Riemannsum}
Let $\varphi$ be a real-valued function, and let $a, b \in \mathbb R$ with $a < b$. If $\varphi'$ is monotonic and $|\varphi'| < 2\pi$ on $(a,b)$, then
%\begin{equation}
\[\left | \int_a^b e^{i \varphi(x)} dx - \sum_{a < n \le b} e^{i \varphi(n)} \right| \le A,\]
%\end{equation}
where the constant $A$ is independent of $\varphi$, $a$, and $b$. 
\end{lemma}

\begin{proof}[Proof of Proposition \ref{dispersion estimate}]
By the orthonormal basis $\{e(x,\xi)\}_{\xi\in\Omega^*}$, we have the representation 
\[e^{it\Delta_\Omega }P_{\le N}w_0 (x) = h^d \sum_{x' \in \Omega} K_{t,N}(x,x')w_0(x'),\]
where 
\begin{equation}\label{kernel for linear propagator}
K_{t,N}(x,x') =~{} \sum_{\substack{\xi \in\Omega^* \\ \max \xi_j \leq \frac{\pi N}{h}}} e^{-i\frac{2t}{h^2}\sum_{j=1}^d (1-\cos(h\xi_j))}e(x,\xi)e(x',\xi).
\end{equation}
Hence, it suffices to show that for all $x,x' \in \Omega$,
\begin{equation}\label{eq:K_t,N}
|K_{t,N}(x,x')| \ls  \left(\frac{N}{|t|h}\right)^{\frac d3}.
\end{equation}
Indeed, factorizing the kernel as 
$$\begin{aligned}
K_{t,N}(x,x') =&~{} \prod_{j=1}^{d} \frac{1}{\pi R} \sum_{0 < \xi_j \le \frac{N\pi}{h}}e^{-i\frac{2t}{h^2} (1-\cos(h\xi_j))}\sin((x_j+\pi R)\xi_j)\sin((x_j'+\pi R)\xi_j)\\
=&~{} \prod_{j=1}^{d} \frac{1}{4\pi R} \sum_{-\frac{N\pi}{h} \le \xi_j \le \frac{N\pi}{h}}e^{-i\frac{2t}{h^2} (1-\cos(h\xi_j))}\big(e^{i(x_j-x_j')\xi_j}-e^{i(x_j+x_j' + 2\pi R)\xi_j}\big),
\end{aligned}$$
the proof of \eqref{eq:K_t,N} is reduced to that of the one-dimensional inequality, 
\begin{equation}\label{eq:K_t,N_1}
\bigg| \frac{1}{4\pi R}\sum_{-\frac{N\pi}{h} \le \xi \le \frac{N\pi}{h}}e^{-i\frac{2t}{h^2} (1-\cos(h\xi))}\big(e^{ix \xi}-e^{i(x + 2\pi R)\xi}\big)\bigg| \lesssim \left(\frac{N}{|t|h}\right)^{\frac 13}.
\end{equation}

To show \eqref{eq:K_t,N_1}, we separate the sum into two pieces, 
\[\begin{aligned}
&\frac{1}{4\pi R}\sum_{-\frac{N\pi}{h} \le \xi \le \frac{N\pi}{h}}e^{-i\frac{2t}{h^2} (1-\cos(h\xi))}\big(e^{ix \xi}-e^{i(x + 2\pi R)\xi}\big)\\
&=\frac{1}{4\pi R} \sum_{-\frac{N\pi}{h} \le \xi \le \frac{N\pi}{h}}e^{-i\frac{2t}{h^2} (1-\cos(h\xi))}e^{ix \xi}-\frac{1}{4\pi R}\sum_{-\frac{N\pi}{h} \le \xi \le \frac{N\pi}{h}}e^{-i\frac{2t}{h^2} (1-\cos(h\xi))}e^{i(x + 2\pi R)\xi}\\
&=:I - II.
\end{aligned}\]
For $I$, we express the frequency variable $\xi$ as $\frac{m}{2R}$, where $m$ is an integer between $-2RNK$ and $2RNK$, and then extracting the approximating integral, we write
\[\begin{aligned}
I &=~{}\frac{1}{4\pi R}\int_{-2RNK}^{2RNK} e^{-it\frac{2}{h^2}(1-\cos(\frac{h \xi}{2R}))}e^{i\frac{x \xi}{2R}} \; d\xi \\
&~{}+\frac{1}{4\pi R} \Bigg\{\sum_{m = -2RNK}^{2RNK} e^{-it\frac{2}{h^2}(1-\cos(\frac{h m}{2R}))}e^{i\frac{x m}{2R}} -\int_{-2RNK}^{2RNK} e^{-it\frac{2}{h^2}(1-\cos(\frac{h \xi}{2R}))}e^{i\frac{x \xi}{2R}} \; d\xi \Bigg\}\\
&=:~{}I_1 + I_2.
\end{aligned}\]
For the integral $I_1$, changing of the variable, we obtain
\[I_1 = \frac{1}{2\pi} \int_{-\frac{N\pi}{h}}^{\frac{N\pi}{h}} e^{-it\frac{2}{h^2}\left(1-\cos\left(h \xi\right)\right)}e^{ix \xi} \; d\xi.\]
We observe that $((\frac{2}{h^2}(1-\cos(h \xi)))'')^2+\frac{1}{h^2}((\frac{2}{h^2}(1-\cos(h \xi)))''')^2=\cos^2(h\xi)+\sin^2(h\xi)=1$. Thus, the van der Corput lemma yields 
\[|I_1| \lesssim\left(\frac{N}{|t|h}\right)^{\frac 13}\]
(see \cite[Lemma 3.3]{Hong2021} for the detailed proof). For the remainder $I_2$, we observe that the phase function $\phi(\xi) = -\frac{2t}{h^2}(1-\cos(\frac{h\xi}{2R})) + \frac{x \xi }{2R}$ obeys $|\phi'(\xi)| < 2\pi$ for all $|\xi| \le 2RNK$, provided that $x \in \Omega$ and $|t| \le \frac{3Rh}{2N}$. Thus, it follows from Lemma \ref{Lem:Riemannsum} that
\[|I_2| \lesssim \frac{1}{R}.\]
Therefore, we conclude that
\[|I| \lesssim \left(\frac{N}{|t|h}\right)^{\frac 13}+\frac{1}{R}.\]

For $II$, we repeat the same argument. Substituting $\xi = \frac{m}{2R}$, we write it as the sum over integers, and then we separating odd $m$'s and even $m$'s as
\[\begin{aligned}
II &= \frac{1}{4\pi R}\sum_{m = -2RNK}^{2RNK} e^{-it\frac{2}{h^2}(1-\cos(\frac{h m}{2R}))}e^{i\frac{x m}{2R}}e^{im\pi}\\
&=\frac{1}{4\pi R}\sum_{m = -RNK}^{RNK} e^{-it\frac{2}{h^2}(1-\cos(\frac{h m}{R}))}e^{i\frac{x m}{R}}-\frac{1}{4\pi R}\sum_{m = -RNK}^{RNK-1} e^{-it\frac{2}{h^2}(1-\cos(\frac{h(2m+1)}{2R}))}e^{i\frac{x (2m+1)}{2R}}.
\end{aligned}\]
Let $\tilde{\phi}_1(\xi) = -\frac{2t}{h^2}(1-\cos(\frac{h\xi}{R})) + \frac{x \xi }{R}$ for the first sum, and let $\tilde{\phi}_2(\xi) = -\frac{2t}{h^2}(1-\cos(\frac{h(2\xi+1)}{2R})) + \frac{x (2\xi+1) }{2R}$ for the second sum. Then, a straightforward computation with $x \in \Omega$ and $|t| \le \frac{Rh}{2N}$ gives $|\tilde{\phi}_1'(\xi)|, |\tilde{\phi}_1'(\xi)| < 2\pi$ for all $|\xi| \le RNK$. Thus, Lemma  \ref{Lem:Riemannsum} implies that the sums in $II$ can be approximated by the corresponding integrals, precisely, 
\[\begin{aligned}
II &=\frac{1}{4\pi R} \int_{-RNK}^{RNK} e^{-it\frac{2}{h^2}(1-\cos(\frac{h \xi}{R}))}e^{i\frac{x \xi}{R}} \; d\xi\\
&\quad-\frac{1}{4\pi R}\int_{-RNK}^{RNK-1} e^{-it\frac{2}{h^2}(1-\cos(\frac{h(2\xi+1)}{2R}))}e^{i\frac{x (2\xi+1)}{2R}} \; d\xi+O\left(\frac{1}{R}\right)\\
&=\frac{1}{4\pi}\int_{-\frac{N\pi}{h}}^{\frac{N\pi}{h}} e^{-it\frac{2}{h^2}\left(1-\cos\left(h \xi\right)\right)}e^{ix \xi} \; d\xi-\frac{1}{4\pi} \int_{-\frac{\pi N}{h}+\frac{1}{2R}}^{\frac{\pi N}{h}-\frac{1}{2R}} e^{-it\frac{2}{h^2}\left(1-\cos\left(h\xi\right)\right)}e^{ix \xi} \; d\xi+O\left(\frac{1}{R}\right).
\end{aligned}\]
Handling the integrals exactly same as $I_1$, we obtain 
\[|II| \lesssim \left(\frac{N}{|t|h}\right)^{\frac 13}+\frac{1}{R}.\]

Finally, we observe that the conditions $|t| \lesssim \frac{Rh}{N}$, $N \ge N_* \sim \frac{h}{\pi}$ and $R\gtrsim 1$ ensure $|t| \lesssim \frac{KR^2}{h}$, thus $\frac1R \lesssim (\frac{N}{R|t|h})^{\frac 13}\lesssim (\frac{N}{|t|h})^{\frac 13}$. Therefore, collecting all, we complete the proof. 
\end{proof}

Next, we prove a time-average $L^\infty$-bound for the linear flow, which will be used to prove the continuum limit.

\begin{proposition}[Uniform time-averaged $L^\infty(\Omega)$-bounds]\label{linear L^infty bound}
Let $d=2, 3$, and let $K, R \in \N$ such that $K = \frac{\pi}{h}$ and $R \sim h^{-\alpha}$ with $\alpha > 0$ for some $h\in(0,1]$. Then, for any sufficiently small $\delta>0$, we have
\begin{equation}\label{eq:Lq}
\|e^{it\Delta_\Omega }w_0\|_{L_t^{\frac{2(1+\min\{\alpha,1\})}{d-2+\delta}}([0,1];L^\infty(\Omega))}\lesssim \|w_0\|_{H^1(\Omega)}.
\end{equation}
\end{proposition}

\begin{remark}\label{L infinity remark1}
On the infinite lattice \cite{Hong2019}, a similar $L^\infty$-bound is obtained simply by combining Strichartz estimates and the Sobolev inequality. However, currently, we cannot apply the same argument to the finite case, because as mentioned in Remark \ref{Bernstein remark}, the Sobolev inequality is partially known in this setting. Fortunately, this technical issue can be detoured slightly modifying the proof of Theorem \ref{thm:Stri}.
\end{remark}

\begin{proof}[Proof of Proposition \ref{linear L^infty bound}]

Recall that previously, the Strichartz estimates \eqref{Strichartz estimate 1} are obtained from the short-time dispersion estimate \eqref{Linftybound:Dyadic}. Now, instead of using \eqref{Linftybound:Dyadic} directly, we interpolate it with a rather trivial bound 
\begin{equation}\label{trivial bound for a frequency piece}
\| e^{it \Delta_\Omega  } P_{N}w_0\|_{L^\infty}\lesssim \left(\frac{N}{h}\right)^{\frac{d}{2}}\| e^{it \Delta_\Omega  } P_{N}w_0\|_{L^2}=\left(\frac{N}{h}\right)^{\frac{d}{2}}\| P_{N}w_0\|_{L^2}\lesssim \left(\frac{N}{h}\right)^d\|w_0\|_{L^1}.
\end{equation}
Then, it follows that if $|t|\leq \frac{Rh}{2N}$, then
$$\| e^{it \Delta_\Omega  } P_{N}w_0\|_{L^\infty}\lesssim \left(\frac{N}{|t|h}\right)^{\frac{d(1-\theta)}{3}}\left(\frac{N}{h}\right)^{d\theta}\|w_0\|_{L^1}=\frac{1}{|t|^{\frac{d(1-\theta)}{3}}}\left(\frac{N}{h}\right)^{\frac{d(1+2\theta)}{3}}\|w_0\|_{L^1}$$
for $0\leq\theta\leq 1$. Thus, applying the argument of Keel and Tao \cite{KT1998}, we prove
\begin{equation}\label{frequency piece L infty bound}
\| e^{it \Delta_\Omega  } P_{N}w_0\|_{L_t^{\frac{6}{d(1-\theta)}}([0,\frac{Rh}{2N}];L^\infty)}\lesssim \left(\frac{N}{h}\right)^{\frac{d(1+2\theta)}{6}}\|w_0\|_{L^2},
\end{equation}
since $(\frac{6}{d(1-\theta)},\infty)$ satisfies the modified admissible equation $\frac{1}{6/d(1-\theta)}+\frac{d(1-\theta)/3}{\infty}=\frac{d(1-\theta)/3}{2}$.

It remains to sum \eqref{frequency piece L infty bound} up. This part is identical to the proof of Theorem \ref{thm:Stri}. If $\alpha\geq 1$ ($\Rightarrow\frac{Rh}{2N}\gtrsim 1$), then taking $\theta=\frac{3}{d}(1-\frac{\delta}{2})-\frac{1}{2}$ and summing in $N$, we obtain
$$\| e^{it \Delta_\Omega  } P_{N}w_0\|_{L_t^{\frac{4}{d-2+\delta}}([0,1];L^\infty)}\lesssim \left(\frac{N}{h}\right)^{1-\frac{\delta}{2}}\|w_0\|_{L^2}.$$
On the other hand, if $0<\alpha<1$, then taking $\theta=\frac{3(2-\delta)}{d(1+\alpha)}-\frac{2-\alpha}{1+\alpha}$ and summing up the time intervals, we obtain
$$\begin{aligned}
\| e^{it \Delta_\Omega  } P_{N}w_0\|_{L_t^{\frac{2(1+\alpha)}{d-2+\delta}}([0,1];L^\infty)}&\lesssim \left(\frac{N}{Rh}\right)^{\frac{d(1-\theta)}{6}}\left(\frac{N}{h}\right)^{\frac{d(1+2\theta)}{6}}\|w_0\|_{L^2}\\
&\sim N^{\frac{(1-\alpha) d(1-\theta)}{6}}\left(\frac{N}{h}\right)^{\frac{d}{6}(2-\alpha+(1+\alpha)\theta)}\|w_0\|_{L^2}\\
&=N^{\frac{(d-2+\delta)(1-\alpha)}{2(1+\alpha)}}\left(\frac{N}{h}\right)^{1-\frac{\delta}{2}}\|w_0\|_{L^2}.
\end{aligned}$$
In both cases, summing in $N$, we complete the proof.
\end{proof}

\begin{remark}\label{L infinity remark2}
The reason we interpolate with \eqref{trivial bound for a frequency piece} in the proof of Proposition \ref{linear L^infty bound} is that in three dimensions, the exponent equation \eqref{r-admissible} coming from \eqref{Linftybound:Dyadic} does not admit the endpoint pair $(2,\infty)$. Interpolation makes different admissible pairs including a pair $(q,r)$ with $r=\infty$, but we have to pay some regularity like when we apply the Sobolev inequality.
\end{remark}

\section{NLS on a finite lattice}\label{sec:6}

We now consider NLS \eqref{eq:DRNLS} on a finite lattice $\Omega$. Indeed, NLS \eqref{eq:DRNLS} is simply a system of finitely coupled ODEs. Thus, given initial data $w_0\in L^2(\Omega)$, the solution $w(t)$ exists in $C([-T,T]; L^2(\Omega))$ for some $T>0$. Then, conservation of the mass and the energy, 
$$M(w)=\|w\|_{L^2(\Omega)}^2\quad\textup{and}\quad E(w)=\frac{1}{2}\|\nabla_h w\|_{L^2(\Omega)}^2+\frac{1}{4}\|w\|_{L^4(\Omega)}^4,$$
again upgrades from local to global well-posedness (see \cite{Hong2019, Hong2021}).

The goal of this section is to obtain uniform bounds for solutions to NLS \eqref{eq:DRNLS}, which are analogous to Proposition \ref{uniform v} and Proposition \ref{prop:xvL2} for the infinite case. First, we prove a time-averaged $L^\infty$ bound.

\begin{proposition}[Uniform time-averaged $L^\infty(\Omega)$-bound ]\label{L^infty bound}
Let $d=2, 3$, and let $K, R \in \N$ such that $K = \frac{\pi}{h}$ and $R \sim h^{-\alpha}$ with $\alpha > 0$ for some $h\in(0,1]$.
% and $h\in(0,1]$ such that $R \sim h^{-\alpha}$, $\alpha \ge 0$. 
We assume that $u_0\in H^1(\Omega)$. Then, the global solution $w(t)$ to NLS  \eqref{eq:DRNLS} with initial data $w_0$ satisfies
\begin{equation}\label{long time L^infty bound}
\|w(t)\|_{L_t^q([-T,T]; L^\infty(\Omega))}\lesssim \langle T\rangle^{\frac1q}, \quad\textup{for all }T>0,
\end{equation}
where $q=\frac{2(1+\min\{\alpha,1\})}{d-2+\delta}$.
\end{proposition}

\begin{proof}
Let $I$ be a (short) interval. Then, applying the unitarity of the linear flow and Proposition \ref{linear L^infty bound} to the Duhamel formula
\begin{equation}\label{eq:Duhamel}
w(t)=e^{it\Delta_\Omega }w_0-i \int_0^t e^{i(t-s)\Delta_\Omega }(|w|^2w)(s) \; ds,
\end{equation}
we obtain
\begin{equation}\label{L^infty bound proof 2}
\|w\|_{L_t^q(I; L^{\infty}(\Omega))} \lesssim \|w_0\|_{H^1(\Omega)} + \||w|^2w\|_{L_t^1(I;H^1(\Omega))}.
\end{equation}
For the nonlinear term, by \eqref{adjoint L^2} and \eqref{eq:product}, we have
$$\begin{aligned}
\||w|^2w\|_{H^1(\Omega)}^2 &= \||w|^2w\|_{L^2(\Omega)}^2+\|\nabla_{\Omega}(|w|^2w)\|_{L^2(\Omega)}^2\\
&\lesssim \|w\|_{L^\infty(\Omega)}^4\|w\|_{L^2(\Omega)}^2+ \|w\|_{L^\infty}^4\|\nabla_{\Omega}w\|_{L^2(\Omega)}^2\\
&= \|w\|_{L^\infty(\Omega)}^4\|w\|_{H^1(\Omega)}^2.
\end{aligned}$$
Hence, it follows that 
\begin{equation}\label{bound for solution_0}
\|w\|_{L_t^q(I; L^{\infty}(\Omega))}
\lesssim \|w_0\|_{H^1(\Omega)} + |I|^{1-\frac{2}{q}}\|w\|_{L_t^{q}(I; L^{\infty}(\Omega))}^2\|w\|_{C_t(I; H^1(\Omega))},
%\le c_0C(M_h,E_h) + c4^{p-1}|\lambda| (2\tau)^{1-\frac{p-1}{q_*}}\left(C(M_h,E_h)\right)^p
\end{equation}
and choosing sufficiently small $|I|$ depending on $\|w_0\|_{H^1(\Omega)}$, we obtain
$$\|w\|_{L_t^q(I; L^{\infty}(\Omega))}\lesssim \|w_0\|_{H^1(\Omega)}.$$
The mass and the energy conservation laws guarantee global-in-time bound on $\|w(t)\|_{H^1(\Omega)}$. Thus, one can iterate for arbitrarily longer time with the bound \eqref{long time L^infty bound} (see \cite[Proposition 4.2]{Hong2021} for the detail).
\end{proof}

Next, we prove the weighted norm bound for the NLS \eqref{eq:DRNLS} on a finite lattice.
\begin{proposition}[Uniform weighted $L^2(\Omega)$-bound]\label{prop:xwL2}
Let $d=2, 3$, and let $K, R \in \N$ such that $K = \frac{\pi}{h}$ and $R \sim h^{-\alpha}$ with $\alpha > 0$ for some $h\in(0,1]$. We assume that $w_0\in H^{1,1}(\Omega)$, and let $w(t)$ be the global solution to NLS \eqref{eq:DRNLS} with initial data $w_0$. Then, there exists $c>0$, depending only on $\|w_0\|_{H^{1,1}(\Omega)}$, such that
$$\|xw(t)\|_{L^2(\Omega)}\lesssim \|w_0\|_{H^{1,1}(\Omega)}e^{ct}\quad\textup{ for all }t\in\mathbb{R}.$$
\end{proposition}

One may attempt to prove the proposition by mimicking the proof of Proposition \ref{prop:xvL2}. However, unlike the infinite case, it is a bit tricky to find the commutator identity corresponding to Lemma \ref{commutator relation hZ}, because multiplication by $x$ breaks the structure of eigenfunctions $e(x,\xi)$'s. Alternatively, we use the following modified weight and commutator inequality.

\begin{lemma}[Commutator inequality]\label{lem:commute}
Let $\varphi = (\varphi_1, \cdots, \varphi_d)$ be a weight function given by
\[\varphi_j = \varphi_j(x_j) = 2R\sin(\tfrac{x_j}{2R}).\]
Then, we have
\begin{equation}\label{weight multiplication}
\varphi_j(x_j) e(x,\xi) = R\Big\{e\big(x,\xi - \tfrac{1}{2R}\mathbf{e}_j\big) - e\big(x,\xi + \tfrac{1}{2R}\mathbf{e}_j\big)\Big\}.
\end{equation}
%where $\{\mathbf{e}_j\}_{j=1}^d$ denotes the standard orthonormal basis of $\R^d$. 
As a consequence, we have
\begin{equation}\label{eq:commute}
\|[\varphi, e^{it\Delta_\Omega }]w_0\|_{L^2(\Omega)} \lesssim |t|\|\sqrt{1-\Delta_\Omega }w_0\|_{L^2(\Omega)} = |t|\|w_0\|_{H^1(\Omega)}.
\end{equation} 
\end{lemma}

\begin{remark}\label{rem:varphi}
$(i)$ The weight function $\varphi$ plays the same role of $x$ as in Lemma \ref{lem:commute} in the proof of Proposition \ref{prop:xwL2}. Note that on $\Omega$, the weight $\phi_j(x_j)$ is comparable with $x_j$. Moreover, $\varphi_j$ acts on $e(x,\xi)$ as a translation by $\pm \frac{1}{2R}$ along the $\xi_j$-direction.\\
$(ii)$ On the right side of \eqref{weight multiplication}, the frequency $\xi \pm \tfrac{1}{2R}\mathbf{e}_j$ may lie on the boundary, i.e., the $j$-th component of $\xi \pm \tfrac{1}{2R}\mathbf{e}_j$ is either 0 or $\frac{\pi}{h}$. In that case, by the definition \eqref{eq: eigenfunction}, we may take $e(\cdot,\xi \pm \tfrac{1}{2R}\mathbf{e}_j)\equiv 0$, and the formula \eqref{weight multiplication} hold properly.
\end{remark}

\begin{proof}[Proof of Lemma \ref{lem:commute}]
By the product-to-sum formula, we have
\[\begin{aligned}
&\varphi_j(x_j) \frac{1}{\sqrt{\pi R}}\sin\left((x_j+\pi R)\xi_j\right) \\
&=\frac{R}{\sqrt{\pi R}} \Big\{\sin\big((x_j+\pi R)(\xi_j - \tfrac{1}{2R})\big) - \sin\big((x_j+\pi R)(\xi_j + \tfrac{1}{2R})\big)\Big\}.
\end{aligned}\]
Thus, \eqref{weight multiplication} follows by the definition \eqref{eq: eigenfunction}.

For \eqref{eq:commute}, we consider the $j$-th component $[\varphi_j, e^{it\Delta_\Omega }]w_0$. Inserting the eigenfunction expansion of $w_0$, it can be expressed as 
\[\begin{aligned}
([\varphi_j, e^{it\Delta_\Omega }] w_0)(x) =&~{} \sum_{\xi \in \Omega^*} \langle w_0, e(\cdot,\xi)\rangle [\varphi_j, e^{it\Delta_\Omega }]e(x,\xi)\\
=&~{} \sum_{\xi \in \Omega^*} \langle w_0, e(\cdot,\xi)\rangle\Big\{\varphi_j(x)e^{-itP_h(\xi)} e(x,\xi) - e^{it\Delta_\Omega }(\varphi_j e(\cdot,\xi))(x)\Big\}.
\end{aligned}\]
Here, by \eqref{weight multiplication}, we rearrange the terms in $\{\cdots\}$ as 
\[\begin{aligned}
&\frac{1}{R}\Big\{\varphi_j(x)e^{-itP_h(\xi)} e(x,\xi) - e^{it\Delta_\Omega }(\varphi_j e(\cdot,\xi))(x)\Big\}\\
&=e^{-itP_h(\xi)} \big(e(x,\xi - \tfrac{1}{2R}\mathbf{e}_j) - e(x,\xi + \tfrac{1}{2R}\mathbf{e}_j)\big)\\
&\quad -  \big(e^{-itP_h(\xi-\frac{1}{2R}\mathbf{e}_j)}e(x,\xi - \tfrac{1}{2R}\mathbf{e}_j)-e^{-itP_h(\xi+\frac{1}{2R}\mathbf{e}_j)}e(x,\xi + \tfrac{1}{2R}\mathbf{e}_j)\big)\\
&=  (e^{-itP_h(\xi+\frac{1}{2R}\mathbf{e}_j)}-e^{-itP_h(\xi)})e(x,\xi + \tfrac{1}{2R}\mathbf{e}_j) +(e^{-itP_h(\xi)}-e^{-itP_h(\xi-\frac{1}{2R}\mathbf{e}_j)})e(x,\xi - \tfrac{1}{2R}\mathbf{e}_j).
\end{aligned}\]
Then, inserting this and recalling that $\{e(x,\xi)\}_{\xi\in\Omega^*}$ forms a basis, we obtain
\[\begin{aligned}
\|[\varphi_j, e^{it\Delta_\Omega }] w_0\|_{L^2(\Omega)}^2&\lesssim  \sum_{\xi \in \Omega^*} |\langle w_0, e(\cdot,\xi)\rangle|^2 R^2\Big\{|e^{-itP_h(\xi+\frac{1}{2R}\mathbf{e}_j)}-e^{-itP_h(\xi}|^2\\
&\qquad\qquad\qquad\qquad\quad +|e^{-itP_h(\xi)}-e^{-itP_h(\xi-\frac{1}{2R}\mathbf{e}_j)}|^2\Big\}.
\end{aligned}\]
By direct computations, we estimate 
\[\begin{aligned}
&\big|e^{-it\frac{2}{h^2}(1-\cos(h(\xi_j \pm \frac{1}{2R})))}- e^{-it\frac{2}{h^2}(1- \cos(h\xi_j))}\big|\\
&\leq\frac{2|t|}{h^2}|\cos(h(\xi_j \pm \tfrac{1}{2R})) - \cos(h\xi_j)|=\frac{4|t|}{h^2}|\sin(h(\xi_j \pm \tfrac{1}{4R})) \sin(\tfrac{h}{4R})|\\
&\leq\frac{|t|}{Rh}|\sin(h(\xi_j \pm \tfrac{1}{4R}))|=\frac{|t|}{Rh}\big|\sin(h\xi_j)\cos(\tfrac{h}{4R})\pm\cos(h\xi_j)\sin(\tfrac{h}{4R})\big|\\
&\leq \frac{|t|}{Rh}|\sin(h\xi_j)|+\frac{|t|}{4R^2}\lesssim \frac{|t|}{R}\sqrt{1+\sin^2(\tfrac{h\xi_j}{2})}
\end{aligned}\]
Inserting this, we obtain
\[\|[\varphi_j, e^{it\Delta_\Omega }] w_0\|_{L^2(\Omega)}^2\lesssim |t|^2\sum_{\xi \in \Omega^*}(1+\sin^2(\tfrac{h\xi_j}{2})) |\langle w_0, e(\cdot,\xi)\rangle|^2.\]
Finally, summing in $j$, we prove \eqref{eq:commute}.
\end{proof}

\begin{proof}[Proof of Proposition \ref{prop:xwL2}]
As mentioned in Remark \ref{rem:varphi}, it suffices to show the same bound for $\varphi w(t)$. To do so, we multiply \eqref{eq:Duhamel} by $\varphi$, and take $L^2(\Omega)$-norm. Then, by Lemma \ref{lem:commute}, we have
\[\begin{aligned}
\|\varphi w(t)\|_{L^2(\Omega)}&\le\| \varphi w_0\|_{L^2(\Omega)}  + \|[\varphi, e^{it\Delta_\Omega }]w_0\|_{L^2(\Omega)}  \\
&\quad+  \int_0^t  \|[\varphi, e^{i(t-s)\Delta_\Omega }](|w|^2w)(s)\|_{L^2(\Omega)}+\| \varphi (|w|^2w)(s)\|_{L^2(\Omega)} \; ds\\
&\lesssim \|\varphi w_0\|_{L^2(\Omega)} + |t|\|w_0\|_{H^1(\Omega)} + \int_0^t |t-s|\|w(s)\|_{L^{\infty}(\Omega)}^2\|w(s)\|_{H^1(\Omega)} \; ds\\
&\quad+\int_0^t \|w(s)\|_{L^{\infty}(\Omega)}^2\|\varphi w(s)\|_{L^2(\Omega)} \; ds.
\end{aligned}\]
Here, by the mass and the energy conservation laws and Proposition \ref{L^infty bound}, we have
$$\int_0^t \|w(s)\|_{L^{\infty}(\Omega)}^2\|w(s)\|_{H^1(\Omega)} \; ds\leq |t|^{1-\frac{2}{q}}\|w(s)\|_{L_s^q([0,t];L^{\infty}(\Omega))}^2\sup_{s\in[0,t]}\|w(s)\|_{H^1(\Omega)} \lesssim |t|.$$
Therefore, by Gr\"onwall’s inequality, we can obtain the desired bound for $\|\varphi w(t)\|_{L^2(\Omega)}$.
\end{proof}

\section{Continuum limit: proof of Theorem \ref{thm:main}}\label{sec:7}

Now, we are ready to prove our main theorem. Indeed, as mentioned in Section \ref{sec:outline}, by Theorem \ref{YH theorem} together with boundedness of the linear interpolation operator $\ell_h$, it is enough to consider the difference between $v(t)$ and $\mathcal{E}w(t)$, where $v(t)$ (resp., $w(t)$) is the solution to the NLS \eqref{eq:DNLS} (resp., the NLS \eqref{eq:DRNLS}) with initial data $d_hu_0$ (resp., $\eta_R(d_hu_0)$). Indeed, for  smaller number $\tilde R$ compared with given $R$, particularly ${\tilde R} = R/4$, by Proposition \ref{prop:xvL2} and Proposition \ref{prop:xwL2}, we have 
$$\begin{aligned}
&\|(1-\eta_{\tilde R})v(t)\|_{L^2(h\Z^d)}+\|(1-\eta_{\tilde R})  \mathcal{E}w(t)\|_{L^2(h\Z^d)}\\
&\lesssim R^{-1}\left\{\|xv(t)\|_{L^2(h\Z^d)}+\|xw(t)\|_{L^2(\Omega)}\right\}\\
&\lesssim R^{-1}e^{ct}\left\{\|d_hu_0\|_{H^{1,1}(h\Z^d)}+\|\eta_{\tilde R}(d_hu_0)\|_{H^{1,1}(\Omega)}\right\}\\
&\lesssim h^{\alpha}e^{ct}\|u_0\|_{H^{1,1}}.
\end{aligned}$$
%where ${\tilde R}$ is a smaller number compared with given $R$, particularly, we choose ${\tilde R} = R/4$. The choice of smaller support $\eta_{\tilde R}$ enables to remove linear part of difference of two solutions inside support, see \eqref{eq:removable}. 
Thus, it remains to consider the localized difference $\eta_{\tilde R} (v-w)(t)=\eta_{\tilde R} (v(t)-\mathcal{E}w(t))$. 

By direct calculations, we observe that $\eta_{\tilde R} (v(t)-w(t))$ solves
\[\begin{aligned}
i\partial_t \eta_{{\tilde R}} (v - w)&= \eta_{\tilde R} \left(- \Delta_h v + |v|^2v + \Delta_\Omega  w  - |w|^2w\right)\\
&= -\Delta_h (\eta_{\tilde R} (v-w)) + \eta_{\tilde R}(|v|^2v - |w|^2w)+ [\Delta_h, \eta_{\tilde R}]v - [\Delta_\Omega , \eta_{\tilde R}]w,
\end{aligned}\]
since two operators $\Delta_h$ and $\Delta_\Omega $ are undistinguishable on the support of $\eta_{\tilde R}$. Moreover, by construction, 
%\begin{equation}\label{eq:removable}
$(\eta_{\tilde R} (v-w))(0)\equiv 0.$
%\end{equation}
 Thus, by Duhamel's principle, we have
\[\eta_{\tilde R} (v - w)(t)=-i\int_0^t e^{i(t-s)\Delta_\Omega }\Big\{|v|^2v - |w|^2w+[\Delta_h, \eta_{\tilde R}]v - [\Delta_\Omega , \eta_{\tilde R}]w\Big\} (s)ds,\]
and thus, 
\[\begin{aligned}
\|\eta_{\tilde R} (v - w)(t)\|_{L^2(\Omega)}&\leq \int_0^t\|\eta_{\tilde R}(|v|^2v - |w|^2w) (s)\|_{L^2(\Omega))}ds\\
&\quad+\int_0^t\|[\Delta_h, \eta_{\tilde R}]v - [\Delta_\Omega , \eta_{\tilde R}]w (s)\|_{L^2(\Omega))}ds\\
&=:I+II.
\end{aligned}\]
For $I$, we have 
\[I \lesssim \int_0^t \left\{\|v(s)\|_{L^{\infty}(h\Z^d)}^2 + \|w(s)\|_{L^{\infty}(\Omega)}^2\right\}\|\eta_{\tilde R}(v - w)(s)\|_{L^2(\Omega)} ds.\]
because $|v|^2v - |w|^2w = (|v|^2 + |w|^2)(v -w) + vw(\overline{v} -\overline{w})$, For $II$, we observe that 
\[\begin{aligned}
\left[\Delta_h, \eta_{\tilde R}\right] v &=\Delta_h(\eta_{\tilde R} v) - \eta_{\tilde R}(\Delta_h v)\\
&= v (\Delta_h \eta_{\tilde R}) +\nabla_h\eta_{\tilde R}\cdot\nabla_h v+\nabla_h^*\eta_{\tilde R}\cdot\nabla_h^* v,
\end{aligned}\]
where $\nabla_h$ is defined by \eqref{right difference gradient} and $\nabla_h^*$ is its adjoint. Thus, by the mass and the energy conservation laws, it follows that 
$$II\leq\int_0^t \|(\Delta_h \eta_{\tilde R})\|_{L^\infty(\Omega)}\|v(s)\|_{L^2(\Omega)} +2\|\nabla_h\eta_{\tilde R}\|_{L^\infty(\Omega)}\|\nabla_h v(s)\|_{L^2(\Omega)}ds\lesssim R^{-1} t\sim h^\alpha t.$$
Collecting all and using Gr\"onwall's inequality, we obtain 
\[
\|\eta_{\tilde R}(v - w)(t)\|_{L^2(\Omega)} \lesssim h^{\alpha} t\exp\left\{\int_0^t \|v(s)\|_{L^{\infty}(h\Z^d)}^2 + \|w(s)\|_{L^{\infty}(\Omega)}^2 ds\right\}.\]
Finally, applying Proposition \ref{uniform v} and \ref{L^infty bound}, we complete the proof.

\appendix

\section{Proof of spectral properties of the discrete Laplacian on a finite cubic lattice}\label{proof of spectral properties}

This appendix provides the proof of Lemma \ref{spectral properties}. A direct computation with the sum-to-product formula for trigonometric functions yields
$$\begin{aligned}
&-\Delta_\Omega  e(x,\xi)\\
&=-\frac{1}{(\pi R)^{\frac{d}{2}}}\sum_{k=1}^d\prod_{j\neq k}\frac{\sin((x_k+h+\pi R)\xi_k)+\sin((x_k-h+\pi R)\xi_k)-2\sin((x_k+\pi R)\xi_j)}{h^2}\\
&=-\frac{1}{(\pi R)^{\frac{d}{2}}}\sum_{k=1}^d\prod_{j=1}^d\sin((x_j+\pi R)\xi_j)\frac{2(\cos(h\xi_k)-1)}{h^2}\\
&=\left\{\frac{2}{h^2}\sum_{k=1}^d1-\cos (h\xi_k)\right\}e(x,\xi).
\end{aligned}$$
To prove that $\{e(x,\xi)\}_{\xi\in\Omega^*}$ is an orthonormal basis, it suffices to show orthonormality of the elements, because $\Omega$ and $\Omega^*$ have the same $(2KR-1)^d$ dimensions. Moreover, since the inner product $\langle e(\cdot, \xi), e(\cdot, \xi')\rangle$ can be factorized as 
\[\begin{aligned}
\langle e(\cdot, \xi), e(\cdot, \xi')\rangle =&~{} \frac{h^d}{(\pi R)^d}\sum_{x \in \Omega} \prod_{j=1}^{d} \sin\left((x_j + \pi R) \xi_j\right)\sin\left((x_j + \pi R) \xi_j'\right)\\
=&~{} \prod_{j=1}^d\left\{\frac{h}{\pi R} \sum_{-\pi R \le x_j < \pi R} \sin\left((x_j + \pi R) \xi_j\right)\sin\left((x_j + \pi R) \xi_j'\right)\right\},
\end{aligned}\] 
it can be further reduced to the one-dimensional case, 
\begin{equation}\label{reduced orthonormality}
\frac{h}{\pi R} \sum_{-\pi R \le x < \pi R} \sin\left((x + \pi R) \xi\right)\sin\left((x + \pi R) \xi'\right)=\delta_{\xi,\xi'},
\end{equation}
where $\delta_{\xi,\xi'}=1$ if $\xi=\xi'$ while it is zero otherwise.

The following elementary summation formula is useful. We omit the proof, because it is straightforward.
\begin{lemma}\label{lem:sumcos}
For any $\theta\in(0,2\pi)$, we have
$$\sum_{j=1}^{L} \cos (j\theta) = \frac{\sin(\frac{(2L+1)\theta}{2})}{2\sin(\frac{\theta}{2})} - \frac12.$$
\end{lemma}

Suppose that $\xi = \xi'$. Then, by the half-angle formula, 
\[\begin{aligned}
\frac{h}{\pi R}\sum_{-\pi R \le x < \pi R}\sin^2\left((x + \pi R) \xi\right) &=\frac{h}{\pi R}\sum_{-\pi R \le x < \pi R} \frac{1 - \cos\left(2(x+\pi R)\xi\right)}{2}\\
&=1-\frac{h}{2\pi R}\sum_{-\pi R \le x < \pi R} \cos\left(2(x+\pi R)\xi\right).
\end{aligned}\]
Hence, it is enough to show
$$\sum_{-\pi R \le x < \pi R} \cos\left(2(x+\pi R)\xi\right)=0.$$
Substituting $\xi = \frac{m}{2R}$ with $m = 1,2, \cdots, 2KR-1$ and $x = \frac{\pi n}{K}$ with $n = -KR, -KR +1, \cdots, KR-2, KR-1$, we write
\[\cos\left(2(x+\pi R)\xi\right) = \cos\left(m\pi + \frac{m \pi}{KR} n\right)=\cos(m\pi+n\theta),\] 
where $\theta = \frac{m \pi}{KR}\in (0,2\pi)$. Therefore, it follows from Lemma \ref{lem:sumcos} that 
$$\begin{aligned}
\sum_{-\pi R \le x < \pi R} \cos\left(2(x+\pi R)\xi\right) &= \cos 0 + \cos (m \pi) + 2\sum_{j=1}^{KR-1} \cos (j\theta)\\
&=1 + \cos (m \pi) + \frac{\sin((KR-1)\theta + \frac{\theta}{2})}{\sin\left(\frac{\theta}{2}\right)} - 1\\
&=\cos (m \pi)+\frac{\sin(m\pi-\frac{\theta}{2})}{\sin\left(\frac{\theta}{2}\right)}=0
\end{aligned}$$
for both even and odd $m$.

Now we assume that $\xi \neq \xi'$, and we aim to show that the left hand side of \eqref{reduced orthonormality} is zero. For $\xi = \frac{m}{2R}$ and $\xi' = \frac{m'}{2R}$, by the product-to-sum formula, we write
\begin{equation}\label{eq:xineqxi'}
\begin{aligned}
&\sum_{-\pi R \le x < \pi R} \sin((x + \pi R) \xi)\sin((x + \pi R) \xi')\\
&=\frac{1}{2}\sum_{-\pi R \le x < \pi R} \cos((x + \pi R) (\xi-\xi'))-\cos((x + \pi R) (\xi+\xi'))\\
&=\frac{1}{2}\sum_{n = -KR}^{KR-1} \cos\big(\tfrac{(m-m')\pi}{2} + \tfrac{(m-m')\pi}{2KR}n\big) - \cos\big(\tfrac{(m+m')\pi}{2} + \tfrac{(m+m')\pi}{2KR}n\big).
\end{aligned}
\end{equation}
When $m \pm m'$ is odd, let $2\theta^{\pm} = (m\pm m')\pi$. Then, we have
\[\begin{aligned}
&\sum_{n = -KR}^{KR-1} \cos\big(\tfrac{(m\pm m')\pi}{2} + \tfrac{(m\pm m')\pi}{2KR}n\big) \\
&=~{} \cos(0) + \cos(\tfrac{\theta^{\pm}}{KR}) + \cdots + \cos\left(\theta^{\pm}\right)+ \cdots +\cos(2\theta^{\pm} - \tfrac{\theta^{\pm}}{KR}) \\
&=~{} \cos(0) + \cos(\theta^{\pm}) = 1,
\end{aligned}\]
due to $\cos(2\theta^{\pm} - \frac{k\theta^{\pm}}{KR}) = - \cos(\frac{k\theta^{\pm}}{KR})$ for $k=1,2, \cdots, KR-1$. Thus, $\eqref{eq:xineqxi'} = 0$. When $m \pm m'=2k$ is even, let $\theta^{\pm} = \frac{(m\pm m')\pi}{2KR}$. Then, $0 < \theta^{\pm} < 2\pi$, and we have from Lemma \ref{lem:sumcos} that
\[\begin{aligned}
\sum_{n = -KR}^{KR-1} \cos\big(\tfrac{(m\pm m')\pi}{2} + \tfrac{(m\pm m')\pi}{2KR}n\big)&=~{}  \cos(0) + \cos(KR \theta^{\pm}) + 2\sum_{j=1}^{KR-1} \cos (j\theta^{\pm})\\
&=~{}  \cos(0) +\cos(k \pi)+ \frac{\sin\big((KR-1)\theta^{\pm} + \frac{\theta^{\pm}}{2}\big)}{\sin(\frac{\theta^{\pm}}{2})} - 1\\
&=~{}  \cos(k \pi)+ \frac{\sin(k \pi  - \frac{\theta^{\pm}}{2})}{\sin(\frac{\theta^{\pm}}{2})} = 0,
\end{aligned}\]
since $\cos(k \pi) = 1$ and $\sin(k \pi  - \frac{\theta^{\pm}}{2}) = -\sin(\frac{\theta^{\pm}}{2})$ if $k$ is even, otherwise $\cos(k \pi) = -1$ and $\sin(k \pi  - \frac{\theta^{\pm}}{2}) = \sin(\frac{\theta^{\pm}}{2})$. 

\section{Small amplitude limit}\label{sec:small amplitude limit}
This appendix is devoted to deducing the small amplitude limit of NLS \eqref{eq:NLS} solution for discrete NLS \eqref{eq:DRNLS} solution on a (bounded) lattice with fixed grid size from Theorem \ref{thm:main}. For a solution $w$ to discrete NLS \eqref{eq:DRNLS}, a simple scaling back yields that
\begin{equation}\label{eq:w_h}
w_h(t,x) = hw(h^2 t, hx)
\end{equation}
is a solution to 
\[\begin{cases}
i\partial_t w_h +\Delta_{\tilde{\Omega}}w_h - |w_h|^2w_h=0,\\
w_h(0) = w_{h,0} = hw_0(hx),
\end{cases} \quad (t,x) \in \R \times \tilde{\Omega},\]
where 
\[\tilde{\Omega} = \tilde{\Omega}_{h,\pi,R} := \left\{x=(x_1,\cdots,x_d)\in \mathbb{Z}^d : -KR \leq x_j\leq KR\right\}\]
and the the second order finite difference $\Delta_{\tilde{\Omega}}$ is defined for a function on $\tilde{\Omega}$ by
\[\Delta_{\tilde{\Omega}} f(x) = \left\{\begin{aligned}&\sum_{j=1}^df(x+\mathbf{e}_j)+f(x-\mathbf{e}_j)-2f(x) &&\textup{for } x \in \tilde{\Omega}\setminus\partial\tilde{\Omega}, \\ &0 &&\textup{for }x \in\partial\tilde{\Omega},\end{aligned}\right.\]
Here $\partial\tilde{\Omega}$ is analogously defined as in \eqref{eq:boundary}. Note that $\tilde{\Omega} \to \Z^d$ as $h \to 0$. On the other hand, we denote the extension operator by $tilde{\mathcal E}$ defined exactly same as in \eqref{eq: extension}, but from $\tilde{\Omega}$ to $\Z^d$. Then, by change of variables, we immediately obtain from \eqref{eq:main result} that
\[\begin{aligned}
\|\ell_1\tilde{\mathcal E} w_h(t) - u_h(t)\|_{L^2(\R^d)} =&~{} h^{\frac d2}\|\ell_h\mathcal E w(h^2t) - u(h^2t)\|_{L^2(\R^d)}\\
\le&~{} Ch^{\frac{d}{2}+\min\{\alpha, \frac12\}}e^{ch^2 t},
\end{aligned}\]
where $u_h$ is another solution to NLS \eqref{eq:NLS} immediately obtained from scaling symmetry, precisely defined by $u_h(t,x) = hu(h^2 t, hx)$. Summarizing above, we conclude that

\begin{theorem}[Small amplitude limit]\label{thm:SAL}
Let $d=2,3$. For $h\in(0,1]$ with $K = \frac{\pi}{h}\in\mathbb{N}$, we take $R \sim h^{-\alpha}$ for some $\alpha > 0$. Given $u_0\in H^{1,1}(\mathbb{R}^d)$, let $u(t)$ be the global solution to NLS \eqref{eq:NLS} with initial data $u_0$, and let $w(t)$ be the global solution to discrete NLS \eqref{eq:DRNLS} with initial data $\eta_R(d_hu_0)$. For $w_h$ defined in \eqref{eq:w_h}, there exists $C, c>0$, depending only on $\|u_0\|_{H^{1,1}(\mathbb{R}^d)}$ not on $h$, such that
\[\|\ell_1 \tilde{\mathcal E} w_h(t) - hu(h^2t, hx)\|_{L^2(\R^d)} \le Ch^{\frac{d}{2}+\min\{\alpha, \frac12\}}e^{ch^2 t} \quad \mbox{for all} \;\; t \in \R.\]
\end{theorem}

\begin{remark}
Even though both small amplitude limit (Theorem \ref{thm:SAL}) and continuum limit are equivalent mathematically, Theorem \ref{thm:SAL} can be interpreted as a stronger result than Theorem \ref{thm:main}. Because, for fixed $t \in \R$, fewer data(larger grid size $h$) are enough for an experiment result to approximate to continuous one with the same error in numerical perspective. Moreover, with the same amount of data (fixed grid size $h$), an experiment result will be available for longer time $t$. 
\end{remark}

\end{document}